\documentclass{mcom-l}

\usepackage{amssymb}
\usepackage{graphicx}
\usepackage{ulem}
\usepackage{multirow}
\usepackage{cite} 
\usepackage{bm}
\usepackage{algorithm}
\usepackage{algpseudocode}
\usepackage{algorithmicx}
\usepackage{setspace}
\usepackage{todonotes}
\usepackage[mathscr]{euscript}
\usepackage{subcaption}
\usepackage{contour}
\DeclareMathAlphabet{\mathpzc}{OT1}{pzc}{m}{it}
\usepackage{url}
\usepackage{hyperref}
\hypersetup{
    colorlinks,
    linkcolor={red!70!black},
    citecolor={green!60!black},
    urlcolor={blue!80!black}
}
\usepackage[nameinlink]{cleveref}

\DeclareMathAlphabet{\mathbf}{OT1}{cmr}{bx}{it}

\newcommand{\spK}{\mathcal{K}}

\newcommand{\R}{\mathbb{R}}

\newcommand{\Rn}{\mathbb{R}^n}
\newcommand{\Rnn}{\mathbb{R}^{n \times n}}

\newcommand{\C}{\mathbb{C}}

\newcommand{\V}{\mathcal{V}}

\DeclareMathOperator{\vecop}{vec}

\DeclareMathOperator{\nnz}{nnz}

\DeclareMathOperator{\dist}{dist}

\newcommand{\blockvec}[1]{#1}                                 
\newcommand{\mat}[1]{\boldsymbol{#1}}                         
\newcommand{\fullbasis}[1]{\boldsymbol{\mathscr{#1}}}         
\newcommand{\fullhessenberg}[1]{\boldsymbol{\mathscr{#1}}}    
\newcommand{\fullcompressedsol}[1]{\boldsymbol{\mathscr{#1}}} 
\newcommand{\blockvecfull}[1]{\mathpzc{#1}}                   
\newcommand{\truncbasis}[1]{\boldsymbol{#1}}                  
\newcommand{\trunchessenberg}[1]{\mat{#1}}                    
\newcommand{\blockscalar}[1]{\bm{#1}}                         
\newcommand{\sPKSM}{\textsc{s-PSKM}}
\newcommand{\tPKSM}{\textsc{t-PSKM}}

\newcommand{\tol}{\texttt{tol}}

\newcommand{\schur}[1]{\mathcal{#1}}

\newcommand{\boldpzc}[1]{\contour[2]{black}{$\mathpzc{#1}$}}

\makeatletter 
\@mparswitchfalse%
\makeatother
\normalmarginpar

\newtheorem{theorem}{Theorem}[section]
\newtheorem{lemma}[theorem]{Lemma}

\theoremstyle{definition}

\newtheorem{example}[theorem]{Example}

\newtheorem{proposition}[theorem]{Proposition}
\newtheorem{corollary}[theorem]{Corollary}

\theoremstyle{remark}
\newtheorem{remark}[theorem]{Remark}

\numberwithin{equation}{section}

\newtheorem{experiment}[theorem]{Experiment}
\let\oldexperiment\experiment
\renewcommand{\experiment}{\oldexperiment\normalfont}

\begin{document}

\title[Sketched Krylov methods for matrix Sylvester equations]{Sketched and truncated polynomial Krylov subspace methods: matrix Sylvester equations}

\author[D. Palitta]{Davide Palitta}
\address{Dipartimento di Matematica, (AM)$^2$, Alma Mater Studiorum - Università di Bologna, 40126 Bologna, Italy}
\email{davide.palitta@unibo.it}

\author[M. Schweitzer]{Marcel Schweitzer}
\address{School of Mathematics and Natural Sciences, Bergische Universit\"at Wuppertal, 42097 Wuppertal, Germany}
\email{marcel@uni-wuppertal.de}

\author[V. Simoncini]{Valeria Simoncini}
\address{Dipartimento di Matematica, (AM)$^2$, Alma Mater Studiorum - Università di Bologna, 40126 Bologna, and IMATI-CNR, Pavia,  Italy}
\email{valeria.simoncini@unibo.it}


\date{}

\dedicatory{}

\begin{abstract}
Thanks to its great potential in reducing both computational cost and memory requirements, combining sketching and Krylov subspace techniques has attracted a lot of attention in the recent literature on projection methods for linear systems, matrix function approximations, and eigenvalue problems. Applying this appealing strategy in the context of linear matrix equations turns out to be far more involved than a straightforward generalization. These difficulties include {\color{black}analyzing} well-posedness of the projected problem and deriving possible error estimates depending on the sketching properties.
Further computational complications include the lack of a natural residual norm estimate and of an explicit basis for the generated subspace.

In this paper we propose a new sketched-and-truncated polynomial Krylov subspace method for Sylvester equations that aims to address all these issues.
The potential of our novel approach, in terms of both computational time and storage demand, is illustrated with numerical experiments. Comparisons with a state-of-the-art projection scheme based on rational Krylov subspaces are also included.
\end{abstract}

\maketitle

\renewcommand{\thefootnote}{\arabic{footnote}}
\section{Introduction}
In the recent literature sketching has been successfully combined with Krylov subspace methods for the solution of linear systems of equations and eigenvalue problems~\cite{balabanov2021randomized,balabanov2022randomized,NakatsukasaTropp2021,TimsitGrigoriBalabanov2023} as well as the numerical approximation of the action of a matrix function on a vector~\cite{PalittaSchweitzerSimoncini2023,CortinovisKressnerNakatsukasa2022,GuettelSchweitzer2022}. In this paper we address another central problem in numerical linear algebra, namely the solution of large-scale Sylvester matrix equations of the form 
\begin{equation}\label{eq:sylvester_equation}
    \mat{A} \mat{X} + \mat{X}\mat{B} = \blockvec{C}_1\blockvec{C}_2^\top,
\end{equation}
where $\mat{A} \in \R^{n_1 \times n_1}$, $\mat{B} \in \R^{n_2 \times n_2}$, and the right-hand side is supposed to be of low rank, i.e., $\blockvec{C}_1 \in \R^{n_1 \times r}$, $\blockvec{C}_2 \in \R^{n_2 \times r}$ with $r \ll \max\{n_1,n_2\}$. The low rank of the right-hand side, along with suitable additional conditions on the coefficient matrices, like, e.g., a large enough relative distance between the spectra of $\mat{A}$ and $-\mat{B}$, implies that the desired solution $\mat{X}$ admits accurate low-rank approximations; see, e.g.,~\cite{BakerEmbreeSabino2015,Grasedyck2004,Penzl2000}.

For the sake of simplicity, in the following we focus on the case $n\equiv n_1=n_2$, namely the solution $\mat{X}$ is an $n\times n$ matrix, but everything we present here can be straightforwardly generalized to the case $n_1\neq n_2$.

The Sylvester equation~\eqref{eq:sylvester_equation} can be encountered in a variety of applications. For instance, several model order reduction techniques give rise to Sylvester equations that feature very large and sparse coefficients $\mat{A}$ and $\mat{B}$~\cite{Antoulas2005}. Similarly, the discretization of certain elliptic and parabolic PDEs yields Sylvester equations of the form~\eqref{eq:sylvester_equation}; see, e.g., \cite{Palitta2016,Palitta2021}.
Another source of applications are linearizations of nonlinear problems, such as the algebraic Riccati equation, where, in general, the solution of~\eqref{eq:sylvester_equation} is needed in iterative solvers~\cite{BiniIannazzoMeini2011}. We refer to the {\color{black}surveys~\cite{Simoncini2016,BennerSaak2013}} for further applications and details.

Our goal is to design an efficient projection procedure for the solution of~\eqref{eq:sylvester_equation} based on the polynomial Krylov subspaces $\mathcal{K}_d(\mat{A},\blockvec{C}_1)$ and $\mathcal{K}_d(\mat{B}^{\top},\blockvec{C}_2)$, where the Krylov subspace is defined as
\begin{equation}\label{eq:def_Krylov}
    \mathcal{K}_d(\mat{A},\blockvec{C}_1)=\text{Range}\{[\blockvec{C}_1,\mat{A} \blockvec{C}_1,\ldots,\mat{A}^{d-1}\blockvec{C}_1]\},
\end{equation}
and analogously for $\mathcal{K}_d(\mat{B}^{\top},\blockvec{C}_2)$.

Early contributions to matrix equation solvers suggested the use of~\eqref{eq:def_Krylov} as approximation space as well; see the pioneering work in \cite{Saad1990a} for the Lyapunov equation, and later contributions in, e.g., 
\cite{JaimoukhaKasenally1994} for the Sylvester equation.
However, it is well-known that the rate of convergence of plain polynomial Krylov subspace methods for Sylvester equations is rather poor in general, and a significant number of iterations is often needed to converge, thus worsening the performance of the overall solution procedure. Attempts to lower the computational cost per iteration and the memory requirements in case of symmetric $\mat{A}$ and $\mat{B}$ can be found in~\cite{PalittaSimoncini2018}. With the same goals, a restarted procedure is presented in~\cite{KressnerLundMasseiPalitta2021}. These contributions come with their limitations and restrictions, so that in practice, projection schemes based on the extended Krylov subspace 
and the rational Krylov subspace are usually preferred, in spite of their more expensive basis construction; see \cite{Simoncini2016}.

{\color{black} Projection-based methods are not the only procedures available in the literature for solving large-scale Sylvester equations. A non-complete list of different solvers sees low-rank Alternating Direction Implicit (ADI) methods~\cite{ADI_Sylv2009}, sign function methods~\cite{Baur2008}, and hierarchical algorithms~\cite{GrasHack2007}.}
{\color{black}We refer to \cite{BennerSaak2013,Simoncini2016} for more detailed accounts of different approaches available for small and large scale problems.}

In this paper we significantly improve the performance of polynomial Krylov subspace methods for Sylvester equations by combining the projection scheme with randomized sketching~\cite{MartinssonTropp2020,woodruff2014sketching}. In particular, we tackle two of the main computational issues due to the large number of iterations needed by polynomial Krylov subspace methods: (i) the high cost of the orthogonalization step and (ii) the large storage demand due to the full allocation of the basis of~\eqref{eq:def_Krylov}. 

As has been shown for linear systems and matrix functions~\cite{CortinovisKressnerNakatsukasa2022,GuettelSchweitzer2022,PalittaSchweitzerSimoncini2023,NakatsukasaTropp2021}, 
one way to unlock the potential of sketching is to combine it with a truncated Krylov scheme, i.e., a Krylov subspace method where the constructed basis is only locally orthogonal. 
The latter approach dramatically reduces the cost of the orthogonalization step. On the other hand, {\color{black} working with a non-orthogonal basis} often leads to a convergence delay; such phenomenon is well-documented for linear systems~\cite{Simoncini2005} and matrix functions~\cite{EiermannErnst2006}. Sketching comes into play to overcome this drawback. In particular, we show that the use of sketching is able to restore the rate of convergence of the full Krylov subspace method without significantly increasing the overall cost of the truncated technique.
We would like to mention that, to the best of our knowledge, 
plain (non-sketched) truncated polynomial Krylov subspace methods for the solution of matrix equations have also never been proposed in the literature. This can thus be seen as a further contribution of this paper.

Thanks to the truncated orthogonalization step, the whole basis is not required during the basis construction, {\color{black} and it is thus not stored, with significant memory savings}. On the other hand, the full basis is needed to retrieve the final solution.
We avoid {\color{black} allocating} the whole basis by employing a so-called two-pass strategy.
This is a common procedure 
which in our case consists in first running the sketched and truncated method without storing the basis. Then, once the prescribed level of accuracy is met, the solution is constructed by computing its low-rank factors during a second, cheaper truncated Arnoldi step; see, e.g., \cite{Kressner2008}. 

The truncated and sketched Krylov machinery has not been used so far in the solution of matrix equations, and as we will illustrate, its formulation and implementation require special attention to obtain an efficient and reliable method. 

{The remainder of the paper is organized as follows. In \Cref{sec:background_material}, we review the basics of polynomial Krylov methods and randomized sketching, as well as the combination of both. \Cref{sec:sketch} introduces the general framework for sketched and truncated Krylov approximations for Sylvester equations and discusses how they relate to the standard Arnoldi approximation. In \Cref{sec:theory}, we investigate the influence of sketching on the field of values and perform a convergence analysis for Lyapunov equations. In \Cref{sec:algorithms}, we discuss implementation aspects and present an algorithm for the sketched Arnoldi method, before presenting numerical experiments in \Cref{sec:numerical_examples}. Concluding remarks are given in \Cref{sec:conclusions}. \Cref{appendix:algorithm} contains an algorithmic description of a truncated Arnoldi method (without sketching), and \Cref{appendix:embedding} collects auxiliary results on tensorized subspace embeddings.}

We assume that exact arithmetic is employed in all the derivations presented in this paper.

Throughout the manuscript, we use the following notation. Vectors are denoted by regular lowercase letters and block vectors (i.e., thin matrices with $r$ columns) are denoted by regular uppercase letters. ``Block scalars'' (i.e., matrices of size $ r\ \times r$ that would be scalars in the rank-1 case) are denoted by bold-face lowercase letters. All other matrices are denoted by bold-face uppercase letters. At some places, calligraphic font is used to help distinguish quantities connected to specific algorithms, which will be explicitly made clear where it is used.  By $\blockvec{E}_i\in{\mathbb R}^{dr\times r}$, we denote the block vector which is all zero except for an identity matrix in its $i$th $r \times r$ block. Given an $n\times n$ real matrix $\mat{M}$, the quantity
$\alpha(\mat{M}) := \lambda_{\max}((\mat{M}+\mat{M}^{\top})/2)$ denotes the rightmost point of the field of values of $\mat{M}$, defined as $W(\mat{M})= \{ \frac{x^* \mat{M}x}{x^*x}, \, {\rm with} \, 0\ne x\in {\mathbb C}^n\}$; here $\lambda_{\max}(\cdot)$ denotes the rightmost eigenvalue of the argument symmetric matrix. We say that $\mat{M}$ is negative definite if $\alpha(\mat{M})<0$.

\section{Background material}\label{sec:background_material}
In this section we review a few known tools that will be  crucial ingredients in what follows.
\subsection{The block Arnoldi relation}
Krylov subspace methods for the solution of Sylvester equations~\eqref{eq:sylvester_equation} build upon the Arnoldi process~\cite{Arnoldi1951}, or a block version of the Arnoldi process (\!\!\cite[Section~6.12]{Saad2003}) if the rank $r$ of the right-hand side $\blockvec{C}_1\blockvec{C}_2^{\top}$ is larger than one. We present all methods in the more general block setting in the following, without writing out the obvious simplifications that are possible when $r=1$. 

Given the coefficient matrices $\mat{A}, \mat{B} \in \Rnn$, the blocks of vectors $\blockvec{C}_1, \blockvec{C}_2 \in \R^{n \times r}$, and number of Krylov iterations\footnote{Note that all our results easily generalize to the case of two different dimensions $d_1$, $d_2$ for the two Krylov spaces, but we restrict to $d_1 = d_2 \equiv d$ to simplify notation.} $d$, the block Arnoldi method computes nested orthonormal bases $\fullbasis{U}_d \in \R^{n \times dr}$ of $\spK_d(\mat{A},\blockvec{C}_1)$ and $\fullbasis{V}_d \in \R^{n \times dr}$ of $\spK_d(\mat{B}^{\top},\blockvec{C}_2)$ via modified block Gram--Schmidt orthonormalization. This yields the \emph{block Arnoldi relations}
\begin{align}\label{eq:arnoldi_relations}
\begin{split}
\mat{A}\fullbasis{U}_d &= \fullbasis{U}_{d+1}\underline{\fullhessenberg{H}}_d = \fullbasis{U}_d \fullhessenberg{H}_d + \blockvecfull{U}_{d+1}\boldpzc{h}_{d+1,d}\blockvec{E}_d^{\top},\\
\mat{B}^{\top}\fullbasis{V}_d &= \fullbasis{V}_{d+1}\underline{\fullhessenberg{G}}_d = \fullbasis{V}_d \fullhessenberg{G}_d + \blockvecfull{V}_{d+1}\boldpzc{g}_{d+1,d}\blockvec{E}_d^{\top},
\end{split}
\end{align}
where 
\[
\underline{\fullhessenberg{H}}_d = 
\begin{bmatrix}
\boldpzc{h}_{1,1} & \boldpzc{h}_{1,2} & \cdots & \boldpzc{h}_{1,d} \\
\boldpzc{h}_{2,1} & \boldpzc{h}_{2,2} & \cdots & \boldpzc{h}_{2,d} \\
                  & \boldpzc{h}_{3,2} & \ddots &  \vdots \\
                  & \ & \ddots &  \boldpzc{h}_{d,d} \\
                  &  &  &  \boldpzc{h}_{d+1,d} \\
\end{bmatrix}, \qquad 
\underline{\fullhessenberg{G}}_d = 
\begin{bmatrix}
\boldpzc{g}_{1,1} & \boldpzc{g}_{1,2} & \cdots & \boldpzc{g}_{1,d} \\
\boldpzc{g}_{2,1} & \boldpzc{g}_{2,2} & \cdots & \boldpzc{g}_{2,d} \\
                  & \boldpzc{g}_{3,2} & \ddots &  \vdots \\
                  & \ & \ddots &  \boldpzc{g}_{d,d} \\
                  &  &  &  \boldpzc{g}_{d+1,d} \\
\end{bmatrix}
\]
are block upper Hessenberg matrices of size $(d+1)r \times {\color{black}d}r$, with blocks $\boldpzc{g}_{i,j}, \boldpzc{h}_{i,j} \in \R^{r \times r}$ {\color{black} and $\fullhessenberg{G}_d,\fullhessenberg{H}_d$ denoting their upper $dr \times dr$ submatrix. Further,} $\fullbasis{U}_{d+1}=[\blockvecfull{U}_1, \dots, \blockvecfull{U}_{d+1}], \fullbasis{V}_{d+1}=[\blockvecfull{V}_1, \dots, \blockvecfull{V}_{d+1}] \in \R^{n \times (d+1)r}$ contain the basis block vectors $\blockvecfull{U}_{i}, \blockvecfull{V}_{i} \in \R^{n \times r}$, respectively.

Performing $d$ steps of the two block Arnoldi schemes requires $2dr$ matrix vector products (half of them with $\mat{A}$ and half of them with $\mat{B}^{\top}$) at a cost of $\mathcal{O}(dr \nnz(\mat{A}) + dr \nnz(\mat{B}))$ floating point operations (flops), assuming that $\mat{A}$ and $\mat{B}$ are sparse with $\nnz(\mat{A})$ and $\nnz(\mat{B})$ nonzeros, respectively. In addition, the modified block Gram--Schmidt orthogonalization needs to be carried out, which in total, across all $d$ iterations, induces a cost of $\mathcal{O}((rd)^2 {\color{black}n})$ flops. This quadratic growth in computational complexity means that with growing iteration number, the orthogonalization step tends to dominate the overall cost of the method. This becomes even more severe on modern high performance architectures where each inner product introduces a global synchronization point requiring communication.

\subsection{Oblivious subspace embeddings}
The basis of sketched Krylov methods is the use of so-called \emph{oblivious subspace embeddings}~\cite{sarlos2006improved,DrineasMahoneyMuthukrishnan2006,woodruff2014sketching}. These allow to embed a (low-dimensional) subspace $\V$ of $\Rn$ into $\mathbb{R}^s$, $s \ll n$, such that norms (or inner products) are distorted in a controlled manner. 

{\color{black} To be specific, for a given $\varepsilon\in [0,1)$, a matrix $\mat{S} \in \R^{s \times n}$ is called an \emph{$\varepsilon$-subspace embedding for $\V$} (or for brevity, simply a \emph{``sketching matrix''}) if it fulfills},
\begin{equation}\label{eq:sketch}
(1-\varepsilon) \| v \|^2 \leq \| \mat{S} v\|^2 \leq (1+\varepsilon) \|v\|^2,
\end{equation}
{\color{black} for all vectors~$v \in \V$}, where $\|\cdot\|$ denotes the Euclidean vector norm. {\color{black} By the parallelogram inequality, this is equivalent to requiring}
\begin{equation}{\label{eq:sketch_innerproduct}}
\langle u, v \rangle - \varepsilon \| u\| \|v\|
            \leq \langle \mat{S} u, \mat{S} v \rangle 
            \leq \langle u, v \rangle + \varepsilon \| u\| \|v\|,
\end{equation}
for all $u,v \in \V$.

{\color{black} In the context of sketched Krylov methods, so-called \emph{oblivious} subspace embeddings are very important. These are embeddings which} can be constructed without explicit knowledge of the subspace $\V$ that shall be embedded. {\color{black} To be precise, $\mat{S}$ is an oblivious $\varepsilon$-subspace embedding for subspaces of dimension $d$, if---with high probability---condition~\eqref{eq:sketch} holds for \emph{any} subspace $\V$ with $\dim(\V) = d$. More precisely, the inputs for constructing $\mat{S}$ are the embedding quality $\varepsilon$, the sketching dimension $s$ and an accepted failure probability $\delta$. Based on these inputs, $\mat{S}$ is then constructed using probabilistic methods. We refer to~\cite[Section~8]{MartinssonTropp2020} and~\cite{woodruff2014sketching} for an extensive treatment of oblivious subspace embeddings. As a specific example, assume that $\mat{S}$ is constructed as a subsampled trigonometric transform,
\[
\mat{S} = \sqrt{\frac{s}{n}}\mat{D}\mat{N}\mat{E},
\]
where $\mat{E} \in \Rnn$ is a diagonal matrix with entries $\pm 1$ chosen randomly (with equal probability), $\mat{D} \in \R^{s \times n}$ contains $s$ randomly selected rows of the identity matrix, and $\mat{N} \in \Rnn$ is a trigonometric transform (e.g., discrete Fourier transform, discrete cosine transform, Walsh--Hadamard transform). Then, one can show that $\mat{S}$ is an oblivious $\varepsilon$-subspace embedding with failure probability $\delta$ if the sketching dimension satisfies $s = \mathcal{O}(\varepsilon^{-2}(d + \log\frac{n}{\delta})\log\frac{d}{\delta})$. When implemented properly, such $\mat{S}$ can be applied at a cost of $\mathcal{O}(n\log s)$ flops.\footnote{When selecting $\mat{S}$ as a matrix with Gaussian random entries, one actually obtains the stronger guarantee that $s = \mathcal{O}(\varepsilon^{-2}(d + \log\frac{1}{\delta}))$ is sufficient, but these are less attractive computationally because they are more costly to apply.} Numerical evidence suggests that, despite certain academic examples for which it fails, selecting $s = \mathcal{O}(\varepsilon^{-2}\frac{d}{\delta})$ works well in practice; see, e.g.,~\cite{tropp2011improved} and \cite[Section~9]{MartinssonTropp2020}.

All computationally feasible classes of matrices from which one can draw $\mat{S}$ have in common that the necessary sketching dimension $s$ depends quadratically on the embedding quality $\varepsilon$. As a consequence} one needs to work with rather crude accuracies, {\color{black}such as} $\varepsilon = 1/\sqrt{2}$ {\color{black}or $\varepsilon = 1/2$, i.e., relatively large distortions in~\eqref{eq:sketch}.}

{\color{black} Oblivious subspace embeddings are key in our context, as the spaces we want to embed will be certain Krylov subspaces (see \Cref{subsec:sketched_arnoldi} and \Cref{sec:sketch} for details), that we do not know a priori,} at the start of the method.

{\color{black} We mention in passing that an alternative viewpoint on subspace embeddings is} that a subspace embedding $\mat{S}$ induces a semidefinite inner product
\begin{equation*}\label{eq:semidef_inner_product}
    \langle u, v\rangle_{\mat{S}} := \langle \mat{S} u, \mat{S} v\rangle,
\end{equation*}
which, when restricted to $\mathcal{V}$, is an actual (positive definite) inner product; see, e.g.,~\cite[Section 3.1]{BalabanovNouy2019}.

{\color{black} This viewpoint is particularly relevant for establishing a connection between our methodology and other randomized Krylov techniques that have been recently proposed in the literature; see the discussion in~\Cref{subsec:sketched_arnoldi}.}

\subsection{The sketched Arnoldi relation}\label{subsec:sketched_arnoldi}
The combination of Krylov subspace methods with randomized oblivious subspace embeddings in order to increase computational efficiency in the orthogonalization step (and potentially reduce storage requirements and communication) is a topic that has recently gained a lot of traction and was investigated in the context of linear systems and eigenvalue computations~\cite{balabanov2021randomized,balabanov2022randomized,NakatsukasaTropp2021} as well as in the approximation of matrix functions~\cite{burke2023krylov,CortinovisKressnerNakatsukasa2022,GuettelSchweitzer2022,PalittaSchweitzerSimoncini2023}. The approach taken in~\cite{GuettelSchweitzer2022,NakatsukasaTropp2021,PalittaSchweitzerSimoncini2023} advocates the use of partially orthogonalized Krylov bases together with a subspace embedding $\mat{S}$ that allows to ``approximately orthogonalize'' the Krylov basis at a much lower cost.

In order to analyze the resulting methods and allow to compare them to their standard, non-sketched counterparts, a \emph{sketched Arnoldi relation} is derived in~\cite{PalittaSchweitzerSimoncini2023}, which we recall in the following. Assume that a truncated (block) Arnoldi procedure is employed for constructing bases of $\spK_d(\mat{A},\blockvec{C}_1)$ and $\spK_d(\mat{B}^{\top},\blockvec{C}_2)$, where orthogonalization {\color{black}only} against the previous {\color{black}$k<d$} blocks is performed, {\color{black}rather than against all $d$ previous blocks, followed by the orthogonalization of the columns of the new block}{\color{black}; see also~\cite[Algorithm~6.6]{Saad2003} for a discussion of the non-block case}. This results in the truncated Arnoldi relations
\begin{equation}\label{eqn:arnoldi_tr}
\mat{A} \truncbasis{U}_d = \truncbasis{U}_{d+1} \underline{\trunchessenberg{H}}_d,\qquad \mat{B}^{\top} \truncbasis{V}_d = \truncbasis{V}_{d+1} \underline{\trunchessenberg{G}}_d,
\end{equation}
where $\underline{\trunchessenberg{H}}_d$ and $\underline{\trunchessenberg{G}}_d$ are upper Hessenberg and banded (with upper bandwidth $kr$), and $\truncbasis{U}_d$, $\truncbasis{V}_d$ do not have all orthogonal columns. We denote $\underline{\trunchessenberg{H}}_d= [\trunchessenberg{H}_d; \blockvec{H}^{\top}]$, $\underline{\trunchessenberg{G}}_d= [\trunchessenberg{G}_d; \blockvec{G}^{\top}]$, where only the last $r \times r$ blocks of $\blockvec{H}^{\top} = \blockscalar{h}_{d+1,d}\blockvec{E}_d^{\top}$ and $\blockvec{G}^{\top} = \blockscalar{g}_{d+1,d}\blockvec{E}_d^{\top}$ are nonzero.\footnote{Following the convention from~\cite{PalittaSchweitzerSimoncini2023}, we use roman font ($\trunchessenberg{H}_d$, $\truncbasis{U}_d$) for quantities related to a non-orthogonal Arnoldi relation and calligraphic font ($\fullhessenberg{H}_d, \fullbasis{U}_d$) for quantities related to a fully orthogonal Arnoldi relation.}

{\color{black} It is known---e.g., from the linear system case~\cite{Simoncini2005}---that basing a projection method on the quantities from~\eqref{eqn:arnoldi_tr} might delay convergence (and the same effect is also observed for matrix equations in our experiments in \Cref{sec:numerical_examples}). One possible remedy for this is to combine the truncated block Arnoldi method with oblivious subspace embeddings in order to cheaply (approximately) orthogonalize the non-orthogonal bases $\mat{U}_{d+1},\mat{V}_{d+1}$, as it was recently proposed for linear systems and matrix function computations in~\cite{NakatsukasaTropp2021,GuettelSchweitzer2022,PalittaSchweitzerSimoncini2023}.

The basis for analyzing this approach is given by the following ``sketched Arnoldi relation'' from~\cite{PalittaSchweitzerSimoncini2023}. We assume that} $\mat{S}_{\truncbasis{U}}$, $\mat{S}_{\truncbasis{V}}$ {\color{black}are} sketching matrices for $\spK_{d+1}(\mat{A},\blockvec{C}_1)$ and $\spK_{d+1}(\mat{B}^{\top},\blockvec{C}_2)$, respectively.

\begin{proposition}\label{prop:Sarnoldi}[Proposition~3.1 in~\cite{PalittaSchweitzerSimoncini2023}, adapted to the block setting] \label{prop:sk_arn_rel}
Let $\mat{S}_{\truncbasis{U}}\truncbasis{U}_{d+1}=\mat{Q}_{\mat{U},d+1} \mat{T}_{\mat{U},d+1}$ be a reduced QR decomposition with 
\[
\mat{Q}_{\mat{U},d+1}=[\mat{Q}_{\mat{U},d},\blockvec{Q}_{\mat{U},d+1}] \text{ and } \mat{T}_{\mat{U},d+1}=
\begin{bmatrix}
\mat{T}_{\mat{U},d}& \blockvec{T}_{H}\\
0^{\top} &\blockscalar{\tau}_{d+1}
\end{bmatrix}.
\]
Then, for the sketched method, the following Arnoldi-like formula holds:
\begin{equation}\label{eq:sketched_arnoldi1}
\mat{S}_{\truncbasis{U}} \mat{A} \truncbasis{U}_d = \mat{S}_{\truncbasis{U}} \truncbasis{U}_{d} (\trunchessenberg{H}_d + \blockvec{R}_{H} \blockvec{E}_d^{\top}) + \blockvec{Q}_{\mat{U},d+1} \blockscalar{\tau}_{d+1} \blockscalar{h}_{d+1,d} \blockvec{E}_d^{\top},
\end{equation}
with $\blockvec{R}_{H}=\mat{T}_{\mat{U},d}^{-1}\blockvec{T}_{H}\blockscalar{h}_{d+1,d}$ and $\blockvec{Q}_{\mat{U},d+1}\perp \mat{S}_{\truncbasis{U}} \truncbasis{U}_d$.
Similarly, if 
$$
\mat{S}_{\truncbasis{V}}\truncbasis{V}_{d+1}=\mat{Q}_{V,d+1} \mat{T}_{V,d+1}, \qquad \mat{T}_{V,d+1}=\begin{bmatrix}
\mat{T}_{V,d} & \blockvec{T}_{G}\\ 
0^{\top} & \blockscalar{\theta}_{d+1}\\
\end{bmatrix},
$$
then
\begin{equation}\label{eq:sketched_arnoldi2}
\mat{S}_{\truncbasis{V}} \mat{B}^{\top}\truncbasis{V}_d = \mat{S}_{\truncbasis{V}} \truncbasis{V}_{d} (\trunchessenberg{G}_d + \blockvec{R}_{G} \blockvec{E}_d^{\top}) +  \blockvec{Q}_{V,d+1}\blockscalar{\theta}_{d+1} \blockscalar{g}_{d+1,d} \blockvec{E}_d^{\top},
\end{equation}
where
$\blockvec{R}_{G}:= \mat{T}_{V,d}^{-1}\blockvec{T}_{G}\blockscalar{g}_{d+1,d}$.
\end{proposition}

The reduced QR decompositions of $\mat{S}_{\truncbasis{U}}\truncbasis{U}_{d+1}$ and
 $\mat{S}_{\truncbasis{V}}\truncbasis{V}_{d+1}$ {\color{black}occurring in \Cref{prop:sk_arn_rel}} can {\color{black}also be used} as a means of orthogonalizing the bases $\truncbasis{U}_d$ and $\truncbasis{V}_d$ with respect to
the inner product $\langle\cdot,\cdot\rangle_{\mat{S}}$, via the transformations $\truncbasis{U}_d \rightarrow \truncbasis{U}_d\mat{T}_{\mat{U},d}^{-1}$ and $\truncbasis{V}_d \rightarrow \truncbasis{V}_d\mat{T}_{V,d}^{-1}$, respectively. 
Thanks to the $\varepsilon$-subspace embedding 
property~\eqref{eq:sketch_innerproduct},  the bases 
\begin{equation}\label{eqn:bases}
\widehat{\truncbasis{U}}_d := \truncbasis{U}_d\mat{T}_{\mat{U},d}^{-1}, \quad 
\widehat{\truncbasis{V}}_d := \truncbasis{V}_d\mat{T}_{V,d}^{-1},
\end{equation}
are {\color{black}very well-behaved}, {\color{black} conditioned to the considered probability}. Indeed, {\color{black} with high probability} it holds
(\!\!\cite[Corollary 2.2]{balabanov2022randomized})
\begin{equation*}\label{eq:whitenend_basis_conditioning}
\kappa_2(\widehat{\truncbasis{U}}_d) \leq \sqrt{\frac{1+\varepsilon}{1-\varepsilon}}, 
\end{equation*}
{\color{black}where $\kappa_2$ denotes the two-norm condition number}, and similarly for $\widehat{\truncbasis{V}}_d$.  
{\color{black}Rewriting the relations~\eqref{eq:sketched_arnoldi1}--\eqref{eq:sketched_arnoldi2} in terms of the bases $\widehat{\truncbasis{U}}_d, \widehat{\truncbasis{V}}_d$ yields the following} {\it whitened-sketched} {\color{black}truncated block} Arnoldi {\color{black}(WS-Arnoldi)} relations (\!\!\cite[Section~3]{PalittaSchweitzerSimoncini2023})
\begin{align}\label{eqn:WSarnoldi}
\begin{split}
\mat{S}_{\truncbasis{U}} \mat{A}\widehat{\truncbasis{U}}_d=
\mat{S}_{\truncbasis{U}}\widehat{\truncbasis{U}}_d (\widehat{\trunchessenberg{H}}_d + \widehat H \blockvec{E}_d^{\top}) +
\blockvec{Q}_{\mat{U},d+1} \blockscalar{h}_{d+1,d} \blockvec{E}_d^{\top}, \\ 
\mat{S}_{\truncbasis{V}} \mat{B}^{\top}\widehat{\truncbasis{V}}_d = \mat{S}_{\truncbasis{V}}\widehat{\truncbasis{V}}_d (\widehat{\trunchessenberg{G}}_d + \widehat G \blockvec{E}_d^{\top}) +   \blockvec{Q}_{V,d+1}\blockscalar{g}_{d+1,d} \blockvec{E}_d^{\top},
\end{split}
\end{align}
where 
\begin{align}
\begin{split}\label{eqn:Hhat_Ghat}
\widehat{\trunchessenberg{H}}_d=\mat{T}_{\mat{U},d} \trunchessenberg{H}_d \mat{T}_{\mat{U},d}^{-1}, &\qquad \widehat{\blockvec{H}} = \blockvec{T}_{H}\blockscalar{h}_{d+1,d}\blockscalar{\tau}_{d}^{-1}, \\
\widehat{\trunchessenberg{G}}_d=\mat{T}_{V,d}\trunchessenberg{G}_d\mat{T}_{V,d}^{-1}, &\qquad \widehat G= \blockvec{T}_{G}\blockscalar{g}_{d+1,d}\blockscalar{\theta}^{-1}_d.
\end{split}
\end{align}
Finally, from~\eqref{eqn:WSarnoldi} we can derive the relations
\begin{align}\label{eqn:WSarnoldi2}
\begin{split}
\widehat{\truncbasis{U}}_d^{\top}\mat{S}_{\truncbasis{U}}^{\top}\mat{S}_{\truncbasis{U}} \mat{A}\widehat{\truncbasis{U}}_d=
\widehat{\trunchessenberg{H}}_d + \widehat{\blockvec{H}} \blockvec{E}_d^{\top}, \qquad
\widehat{\truncbasis{V}}_d^{\top}\mat{S}_{\truncbasis{V}}^{\top}\mat{S}_{\truncbasis{V}} \mat{B}^{\top}\widehat{\truncbasis{V}}_d = \widehat{\trunchessenberg{G}}_d + \widehat{\blockvec{G}} \blockvec{E}_d^{\top},
\end{split}
\end{align}
which emphasize the role of the right-hand side matrices as projection and reduction
of the matrices $\mat{A}$ and $\mat{B}^{\top}$ onto the generated spaces, in the $\mat{S}_{\truncbasis{U}}^{\top}\mat{S}_{\truncbasis{U}}$
and $\mat{S}_{\truncbasis{V}}^{\top}\mat{S}_{\truncbasis{V}}$ inner product, respectively. 
{\color{black}We stress that the full
matrices $\widehat{\truncbasis{U}}_d, \widehat{\truncbasis{V}}_d$ are never explicitly constructed in practice, whereas the small 
 matrices $\widehat{\trunchessenberg{H}}_d, \widehat{\trunchessenberg{G}}_d$ are used to derive the sought after problem approximation.}

{\color{black} We end this section by commenting on the connection between the methodology outlined above and other randomized Krylov methods that were recently proposed, in particular those based on the \emph{randomized Gram--Schmidt (RGS) process}~\cite{balabanov2021randomized,balabanov2022randomized,CortinovisKressnerNakatsukasa2022,TimsitGrigoriBalabanov2023}. In these methods, one uses a Gram-Schmidt procedure---with respect to the sketched inner product~\eqref{eq:semidef_inner_product} instead of the usual Euclidean inner product---for constructing the Krylov basis. In exact arithmetic, this approach yields the same bases as~\eqref{eqn:bases} if the QR decompositions are computed without interchanging columns. In particular, the WS-Arnoldi relations~\eqref{eqn:WSarnoldi} and~\eqref{eqn:WSarnoldi2} are equivalent to relations obtained in~\cite[Section~3.2]{TimsitGrigoriBalabanov2023}, but they are phrased differently and reveal different structural properties.

Thus, as was already noted in~\cite{TimsitGrigoriBalabanov2023}, methods based on the RGS process and those based on truncated Arnoldi and sketching are mathematically equivalent. There are several notable differences, though: Performing the RGS process necessarily requires keeping the whole Krylov basis in memory, and the asymptotic cost of the orthogonalization still scales as $\mathcal{O}((rd)^2n)$. Indeed, while inner products are computed with vectors of length $s$, all other vector operations have the same cost as in standard Gram--Schmidt. In contrast, when using truncated Arnoldi with sketching, only the last $k$ basis blocks need to be kept in memory, as the coefficients for the orthogonalization in the semidefinite inner product~\eqref{eq:semidef_inner_product} are computed from the sketched bases $\mat{S}_{\truncbasis{U}}\truncbasis{U}_{d+1}, \mat{S}_{\truncbasis{V}}\truncbasis{V}_{d+1}$. As we will explain in \Cref{sec:sketch}, the transformations~\eqref{eqn:bases} need not be performed explicitly, so that using this approach allows to use a two-pass strategy for reducing memory consumption; see \Cref{sec:algorithms} below. In summary, our approach reduces the asymptotic cost and memory consumption of the basis construction in comparison to RGS. On the other hand, methods based on RGS can in general be expected to be more stable, as the truncated block Arnoldi process and the retrospective orthogonalization via $\mat{T}_{\mat{U},d}^{-1}$ and $\mat{T}_{V,d}^{-1}$ can potentially be sources of instability. We refer to~\cite{CortinovisKressnerNakatsukasa2022}, where it is illustrated (for the problem of computing $f(A)b$, the action of a matrix function on a vector) that RGS-based methods might be successful {\color{black} in certain examples where} those based on truncated orthogonalization are not. The precise effects causing this difference in performance are not yet fully understood, though. In particular, they are not merely related to ill-conditioning of the bases from truncated Arnoldi: As examples in~\cite{GuettelSchweitzer2022,PalittaSchweitzerSimoncini2023} as well as in \Cref{sec:numerical_examples} below show, the methods often continue to work well even in presence of a numerically singular basis.}

\section{Sketched and truncated Krylov methods for matrix equations}\label{sec:sketch}
 We now turn more specifically to the solution of the 
Sylvester equation~\eqref{eq:sylvester_equation} by means of projection onto
the generated spaces;  see, e.g.,~\cite{Simoncini2016}. We seek an approximation in the form $\mat{X}_d=\truncbasis{U}_d\mat{Y}_d\truncbasis{V}_d^{\top}$, with $\truncbasis{U}_d, \truncbasis{V}_d$ generated
above 
and $\mat{Y}_d\in\mathbb{R}^{dr\times dr}$ computed by solving some \emph{surrogate} problem. 
Using the truncated Arnoldi relations in (\ref{eqn:arnoldi_tr}) and $\mat{X}_d=\truncbasis{U}_d\mat{Y}_d\truncbasis{V}_d^{\top}$, 
the residual matrix $\mat{R}_d=\mat{A} \mat{X}_d+\mat{X}_d\mat{B}-\blockvec{C}_1\blockvec{C}_2^{\top}$ can be written as
 \begin{equation}\label{eq:residualform_po}
 \mat{R}_d=[\truncbasis{U}_d,\blockvec{U}_{d+1}]\begin{bmatrix}
     \trunchessenberg{H}_d\mat{Y}_d+\mat{Y}_d\trunchessenberg{G}_d^{\top}-\blockvec{E}_1{\blockscalar\beta}_1{\blockscalar\beta}_2^{\top}\blockvec{E}_1^{\top} & \mat{Y}_dE_{d}\blockscalar{g}_{d+1,d}^\top\\
     \blockscalar{h}_{d+1,d}\blockvec{E}_d^{\top}\mat{Y}_d & 0 \\
 \end{bmatrix}[\truncbasis{V}_d,\blockvec{V}_{d+1}]^{\top}, 
 \end{equation}
 where $\blockvec{C}_1=\truncbasis{U}_d\blockvec{E}_1{\blockscalar\beta}_1$, $\blockvec{C}_2=\truncbasis{V}_d\blockvec{E}_1{\blockscalar\beta}_2$, 
and 
${\blockscalar\beta}_1,{\blockscalar\beta}_2\in\mathbb{R}^{r\times r}$. 
{\color{black}The matrix $\mat{Y}_d$ can be obtained by, e.g., imposing some optimality condition on the residual matrix, such as orthogonality with respect to the approximation bases, i.e., the so-called Galerkin condition, $\truncbasis{U}_d^{\top} \mat{R}_d \truncbasis{V}_d=0$.}
If $\truncbasis{U}_{d+1}=[\truncbasis{U}_d,\blockvec{U}_{d+1}]$ and $\truncbasis{V}_{d+1}=[\truncbasis{V}_d,\blockvec{V}_{d+1}]$ had orthonormal columns, {\color{black} this condition would correspond to zeroing out the top left term in the inner block matrix in (\ref{eq:residualform_po}),
thus solving 
$\trunchessenberg{H}_d\mat{Y}_d+\mat{Y}_d\trunchessenberg{G}_d^{\top}=\blockvec{E}_1{\blockscalar\beta}_1{\blockscalar\beta}_2^{\top}\blockvec{E}_1^{\top}$ for $\mat{Y}_d$.
Hence, it would also hold that
$\|\mat{R}_d\|_F^2=\|\blockscalar{h}_{d+1,d}\blockvec{E}_d^{\top}\mat{Y}_d \|_F^2+ \| \mat{Y}_dE_{d+1}\blockscalar{g}_{d+1,d}^\top\|_F^2$;
 see, e.g.,~\cite{Simoncini2016}.
In our setting, not all the columns of $\truncbasis{U}_{d+1}$ and 
$\truncbasis{V}_{d+1}$ are orthonormal. Nonetheless, by denoting $\mat{X}_d^\textsc{tr}=\truncbasis{U}_d\mat{Y}_d^\textsc{tr}\truncbasis{V}_d^{\top}$,
we can still compute $\mat{Y}_d^\textsc{tr}$ by solving the same equation, namely
 \begin{equation*}\label{eqn:reduced_tr}
\trunchessenberg{H}_d\mat{Y}_d^\textsc{tr} + \mat{Y}_d^\textsc{tr} \trunchessenberg{G}_d^{\top} = \blockvec{E}_1{\blockscalar \beta}_1{\blockscalar \beta}_2^{\top} \blockvec{E}_1^{\top}.
\end{equation*}
 Notice that, in contrast to what happens in the full Arnoldi method, this matrix equation is not what would be obtained by imposing a Galerkin condition on the residual matrix, due to the non-orthogonality of the bases\footnote{\color{black} The condition $\truncbasis{U}_d^{\top} \mat{R}_d \truncbasis{V}_d=0$  yields
the reduced equation
$\truncbasis{U}_d^{\top} \mat{A}  \truncbasis{U}_d
\mat{Y}_d \truncbasis{V}_d^\top 
\truncbasis{V}_d
+\truncbasis{U}_d^{\top}\truncbasis{U}_d \mat{Y}_d \truncbasis{V}_d^\top
\mat{B} \truncbasis{V}_d
=\truncbasis{U}_d^{\top}C_1 C_2^\top \truncbasis{V}_d$.}. 
On the other hand, it allows us to derive an approximation without explicitly storing the two bases, and a computable estimate} {\color{black}for the residual norm. In particular, for the residual matrix $\mat{R}_d^{\textsc{tr}}$ stemming from the truncated scheme, the following result holds true.

\begin{proposition}
Let $\mat{R}_d^{\textsc{tr}}=\mat{A}\mat{X}_d^{\textsc{tr}}+\mat{X}_d^{\textsc{tr}}\mat{B}-C_1C_2^\top$, $\mat{X}_d^{\textsc{tr}}=\mat{U}_d\mat{Y}_d^{\textsc{tr}}V_d^\top$,  be the residual matrix stemming from the $d$th iteration of the truncated Krylov method. Then
\begin{equation}\label{eqn:restr}
\|\mat{R}_d^{\textsc{tr}}\|_F \le \sqrt{dr}(\|\mat{Y}_d^{\textsc{tr}}E_{d}\blockscalar{g}_{d+1,d}^\top\|_F+ \|\blockscalar{h}_{d+1,d}\blockvec{E}_d^{\top}\mat{Y}_d^{\textsc{tr}}\|_F).
\end{equation}

\end{proposition}

\begin{proof}
The expression in~\eqref{eq:residualform_po} is equivalent to writing $\mat{R}_d^{\textsc{tr}}$ as
$$
\mat{R}_d^{\textsc{tr}}=
\truncbasis{U}_d     \mat{Y}^{\textsc{tr}}_dE_{d}\blockscalar{g}_{d+1,d}^\top\blockvec{V}_{d+1}^{\top}+     \blockvec{U}_{d+1}\blockscalar{h}_{d+1,d}\blockvec{E}_d^{\top}\mat{Y}_d^{\textsc{tr}}\truncbasis{V}_d^{\top}.
$$
Therefore, by recalling that $U_{d+1}^\top U_{d+1}=V_{d+1}^\top V_{d+1}=\mat{I}_r$, it holds
$$
\|\mat{R}_d^{\textsc{tr}}\|_F\le
\|\truncbasis{U}_d\|_F\|\mat{Y}^{\textsc{tr}}_dE_{d}\blockscalar{g}_{d+1,d}^\top\|_F+ \|\truncbasis{V}_d\|_F\|\blockscalar{h}_{d+1,d}\blockvec{E}_d^{\top}\mat{Y}_d^{\textsc{tr}}\|_F.
$$
The result follows by noticing that $\|\mat{U}_d\|_F= \|\mat{V}_d\|_F= \sqrt{dr}$, since $\mat{U}_d, \mat{V}_d\in{\mathbb R^{n\times dr}}$ have unit norm columns. 
\end{proof}

The truncated Krylov method we just outlined is reminiscent of truncated methods for linear systems; 
see, e.g.~\cite[Sections~6.4.2 and~6.5.6]{Saad2003}. To the best of our knowledge, 
these truncated Arnoldi methods have never been used in the matrix equation setting. }

We now turn our attention to how we can integrate sketching and Krylov schemes for matrix equations.
If sketching is employed, we look for an approximate solution $\mat{X}_d^\textsc{sk}=\widehat{\truncbasis{U}}_d \mat{Y}_d^\textsc{sk} \widehat{\truncbasis{V}}_d^{\top}$.
Thanks to the WS-Arnoldi relations~\eqref{eqn:WSarnoldi}, we generalize the previous relation~\eqref{eq:residualform_po} for the residual to get
 { \footnotesize
 \begin{align}\label{eq:skecthedres_expression}
    &\mat{S}_{\truncbasis{U}}\mat{R}_d^\textsc{sk}\mat{S}_{\truncbasis{V}}^{\top} \\
    &=\mat{S}_{\truncbasis{U}}\widehat{\truncbasis{U}}_{d+1}\begin{bmatrix}
     (\widehat{\trunchessenberg{H}}_d+\widehat{\blockvec{H}}\blockvec{E}_d^{\top})\mat{Y}_d^\textsc{sk}+\mat{Y}^\textsc{sk}_d(\widehat{\trunchessenberg{G}}_d+\widehat{\blockvec{G}}E_D^{\top}) ^{\top}-\blockvec{E}_1\blockscalar{\beta}_1\blockscalar{\beta}_2^{\top}\blockvec{E}_1^{\top} & \mat{Y}_d^\textsc{sk}E_{\color{black}d}\blockscalar{g}_{d+1,d}^\top\\
     \blockscalar{h}_{d+1,d}\blockvec{E}_d^{\top}\mat{Y}^\textsc{sk}_d & 0 \\
 \end{bmatrix}\widehat{\truncbasis{V}}_{d+1}^{\top}\mat{S}_{\truncbasis{V}}^{\top}, \nonumber
\end{align}}%
where now $\blockvec{C}_1=\widehat{\truncbasis{U}}_d\blockvec{E}_1\blockscalar{\beta}_1$, $\blockvec{C}_2=\widehat{\truncbasis{V}}_d\blockvec{E}_1\blockscalar{\beta}_2$.
It is thus again natural to compute $\mat{Y}^\textsc{sk}_d$ so as to annihilate the (1,1)-block in the inner
matrix, that is
\begin{equation}\label{eqn:reduced_sk}
(\widehat{\trunchessenberg{H}}_d + \widehat{\blockvec{H}} \blockvec{E}_d^{\top})  \mat{Y}^\textsc{sk}_d +   \mat{Y}^\textsc{sk}_d (\widehat{\trunchessenberg{G}}_d + \widehat{\blockvec{G}}  \blockvec{E}_d^{\top})^{\top} = \blockvec{E}_1{\blockscalar\beta}_1{\blockscalar\beta}_2^{\top}\blockvec{E}_1^{\top}.
\end{equation}
This way  of determining $\mat{Y}^\textsc{sk}_d$ stands on more solid grounds than in the truncated case. Indeed,
the columns of $\widehat{\truncbasis{U}}_d$ and $\widehat{\truncbasis{V}}_d$ are orthogonal bases with 
respect to the sketched inner products defined by $\mat{S}_{\truncbasis{U}}^{\top}\mat{S}_{\truncbasis{U}}$ and $\mat{S}_{\truncbasis{V}}^{\top}\mat{S}_{\truncbasis{V}}$, respectively.
Thus, solving~\eqref{eqn:reduced_sk} is equivalent to imposing the following 
Galerkin $\mat{S}_*^{\top}\!\mat{S}_*$-orthogonality condition on the residual matrix, namely
\begin{equation*}\label{eqn:sres_gal}
\widehat{\truncbasis{U}}_d^{\top} \mat{S}_{\truncbasis{U}}^{\top}\mat{S}_{\truncbasis{U}} \mat{R}_d^\textsc{sk} \mat{S}_{\truncbasis{V}}^{\top}\mat{S}_{\truncbasis{V}} \widehat{\truncbasis{V}}_d = 0.
\end{equation*}
The equivalence follows from the expression~\eqref{eq:skecthedres_expression} 
and the WS-Arnoldi relations~\eqref{eqn:WSarnoldi2}.

For the sketched-and-truncated method, the Frobenius norm of the sketched residual, namely $\|\mat{S}_{\truncbasis{U}}\mat{R}_d^{\textsc{sk}}\mat{S}_{\truncbasis{V}}^{\top}\|_F$, can be
cheaply evaluated as
\begin{equation}\label{eq:resnorm_sk}
    \|\mat{S}_{\truncbasis{U}}\mat{R}_d^{\textsc{sk}}\mat{S}_{\truncbasis{V}}^{\top}\|_F^2= \| \mat{Y}_d^{\textsc{sk}}\blockvec{E}_d\blockscalar{g}_{d+1,d}^\top\|_F^2 + \|\blockscalar{h}_{d+1,d}\blockvec{E}_d^{\top} \mat{Y}_d^{\textsc{sk}}\|_F^2,
\end{equation}    
where  \eqref{eq:skecthedres_expression} was used together with the orthogonality of the bases $\widehat{\truncbasis{U}}_d$ and $\widehat{\truncbasis{V}}_d$ with respect to the sketched inner products.
Indeed, this corresponds to using the  $\mat{S}_*^{\top}\!\mat{S}_*$-norm, supported by 
the fact that this will be close to the Frobenius norm  
thanks to the $\varepsilon$-subspace embedding property. 
More precisely, 
{\color{black} with high probability,}
\begin{equation}\label{eq:sketched_res_norm_bound}
(1+\widetilde{\varepsilon})^{-1} \|\mat{S}_{\truncbasis{U}}\mat{R}_d^{\textsc{sk}}\mat{S}_{\truncbasis{V}}^{\top}\|_F^2\leq 
\|\mat{R}_d^{\textsc{sk}}\|_F^2 \le
(1-\widetilde{\varepsilon})^{-1}
\|\mat{S}_{\truncbasis{U}}\mat{R}_d^{\textsc{sk}}\mat{S}_{\truncbasis{V}}^{\top}\|_F^2,
\end{equation}
where $\widetilde{\varepsilon} = \varepsilon(2+\varepsilon)$; see \Cref{appendix:embedding}.

\subsection{Distance from the Arnoldi approximation}
It is relevant to study the discrepancy between
the solution to the optimal reduced equation and the solution
$\mat{Y}_d^{\textsc{sk}}$ to the sketched reduced equation (\ref{eqn:reduced_sk}). Analogous results can be obtained for the truncated approximation.
The full Galerkin method determines the Arnoldi approximation
$\mat{X}_d^{\textsc{full}}=\fullbasis{U}_d\fullcompressedsol{Y}_d^{\textsc{full}}\fullbasis{V}_d^{\top}$, where $\fullcompressedsol{Y}_d^{\textsc{full}}$ is such that 
\begin{equation}\label{eqn:full}
\fullhessenberg{H}_d\fullcompressedsol{Y}_d^{\textsc{full}}+ \fullcompressedsol{Y}_d^{\textsc{full}} \fullhessenberg{G}_d^{\top}=\blockvec{E}_1\blockscalar{\beta}_1^{\textsc{full}}(\blockscalar{\beta}_2^{\textsc{full}})^{\top}\blockvec{E}_1^{\top},
\end{equation}
with $\fullbasis{U}_d$, $\fullbasis{V}_d$ the orthonormal bases, $\fullhessenberg{G}_d, \fullhessenberg{H}_d$ the projected matrices from~\eqref{eq:arnoldi_relations}, and $C_1=\blockvecfull{U}_{1}\blockscalar{\beta}_1^{\textsc{full}}$, $C_2=\blockvecfull{V}_{1}\blockscalar{\beta}_2^{\textsc{full}}$; see, e.g., \cite{Simoncini2016}.

{
\color{black} Thanks to the sketched Arnoldi relation in \Cref{prop:Sarnoldi}, it is possible to easily} represent the sketched approximation in terms of the orthogonal basis $\fullbasis{U}_d$ instead of the basis $\truncbasis{\widehat U}_d$;
see~\cite[Section~6]{PalittaSchweitzerSimoncini2023}. 
First, note that there exist upper triangular matrices $\fullhessenberg{T}_{\mat{U}},\fullhessenberg{T}_{\mat{V}}$ such that $\truncbasis{\widehat  U}_d = \fullbasis{U}_d \fullhessenberg{T}_{\truncbasis{U}}$ and $\truncbasis{\widehat  V}_d = \fullbasis{V}_d \fullhessenberg{T}_{\truncbasis{V}}$. 
{\color{black}Analogously to how~\eqref{eqn:WSarnoldi} follows from~\eqref{eqn:bases} and~\eqref{eqn:Hhat_Ghat}, we can thus derive the relations}
\begin{eqnarray*}
\fullhessenberg{T}_{\truncbasis{U}} (\trunchessenberg{\widehat H}_d +\widehat HE_d^\top)\fullhessenberg{T}_{\truncbasis{U}}^{-1} &=& \fullhessenberg{H}_d - \fullbasis{U}_d^\top \blockvec{\widehat  U}_{d+1}\blockscalar{h}_{d+1,d} \blockscalar{t}_{U,d}^{-1}E_d^\top ,\\
\fullhessenberg{T}_{\truncbasis{V}} (\trunchessenberg{\widehat  G}_d+\widehat GE_d^\top) \fullhessenberg{T}_{\truncbasis{V}}^{-1} &=& \fullhessenberg{G}_d - \fullbasis{V}_d^\top \blockvec{\widehat  V}_{d+1}\blockscalar{g}_{d+1,d} \blockscalar{t}_{V,d}^{-1}E_d^\top,
\end{eqnarray*}
where $\blockscalar{t}_{U,d}=E_d^\top\fullhessenberg{T}_{\truncbasis{U}}E_d$ and $\blockscalar{t}_{V,d}=E_d^\top\fullhessenberg{T}_{\truncbasis{V}}E_d$.
Therefore, by pre- and postmultiplying~\eqref{eqn:reduced_sk} by $\fullhessenberg{T}_{\truncbasis{U}}$ and $\fullhessenberg{T}_{\truncbasis{V}}^\top$, respectively,
 we can write $\mat{X}_d^{\textsc{sk}}=\fullbasis{U}_d \fullcompressedsol{Y}_d^{\textsc{sk}}\fullbasis{V}_d^{\top}$, where $\fullcompressedsol{Y}_d^{\textsc{sk}}=\fullhessenberg{T}_{\truncbasis{U}}\mat{Y}_d^{\textsc{sk}}\fullhessenberg{T}_{\truncbasis{V}}^\top$ solves
\begin{equation}\label{eq:compressedproblem_sketched_full}
(\fullhessenberg{H}_d - {\widehat{\blockvec{R}}}_HE_d^\top)\fullcompressedsol{Y}_d^{\textsc{sk}} + \fullcompressedsol{Y}_d^{\textsc{sk}}(\fullhessenberg{G}_d - {\widehat{\blockvec{R}}}_GE_d^\top)^\top=\blockvec{E}_1\blockscalar{\beta}_1^{\textsc{full}}(\blockscalar{\beta}_2^{\textsc{full}})^{\top}\blockvec{E}_1^{\top},
\end{equation}
with
${\widehat{\blockvec{R}}}_H:=\fullbasis{U}_d^{\top}\blockvec{\widehat U}_{d+1}\blockscalar{h}_{d+1,d}\blockscalar{t}_{U,d}^{-1}$,
${\widehat{\blockvec{R}}}_G:=\fullbasis{V}_d^{\top}\blockvec{\widehat V}_{d+1}\blockscalar{g}_{d+1,d}\blockscalar{t}_{V,d}^{-1}$. 
The right-hand side is derived as
$\fullhessenberg{T}_{\truncbasis{U}}E_1\blockscalar{\beta}_1=
E_1 \blockscalar{t}_{U,1} \blockscalar{\beta}_1=\blockvec{E}_1\blockscalar{\beta}_1^{\textsc{full}}$. Similarly for $\blockscalar{\beta}_2^{\textsc{full}}$.

Subtracting~\eqref{eq:compressedproblem_sketched_full} from~\eqref{eqn:full} we obtain
$$
\fullhessenberg{H}_d (\fullcompressedsol{Y}_d^{\textsc{full}}-\fullcompressedsol{Y}_d^{\textsc{sk}}) + 
(\fullcompressedsol{Y}_d^{\textsc{full}}-\fullcompressedsol{Y}_d^{\textsc{sk}}) \fullhessenberg{G}_d^{\top} =
-{\widehat{\blockvec{R}}}_H\blockvec{E}_d^{\top} \fullcompressedsol{Y}_d^{\textsc{sk}} - \fullcompressedsol{Y}_d^{\textsc{sk}}\blockvec{E}_d{\widehat{\blockvec{R}}}_G\phantom{)\!\!\!\!\!\!}^{\top}.
$$ 
Hence, {\color{black} for $\mat{\mathcal{L}} = \fullhessenberg{G}_d\otimes \mat{I} + \mat{I} \otimes \fullhessenberg{H}_d $, we have}
$$
{\color{black}\|\mat{X}_d^{\textsc{full}}-\mat{X}_d^{\textsc{sk}}\|_F=}\|\fullcompressedsol{Y}_d^{\textsc{full}}-\fullcompressedsol{Y}_d^{\textsc{sk}}\|_F \le \|\mat{\mathcal{L}}^{-1}\|_F\,
 \|{\widehat{\blockvec{R}}}_H\blockvec{E}_d^{\top} \fullcompressedsol{Y}_d^{\textsc{sk}} + \fullcompressedsol{Y}_d^{\textsc{sk}}\blockvec{E}_d {\widehat{\blockvec{R}}}_G\phantom{)\!\!\!\!\!\!}^{\top}\|_F.
$$

{The above bound estimates the distance $\|\mat{X}_d^{\textsc{full}}-\mat{X}_d^{\textsc{sk}}\|_F$ in terms 
of a quantity associated with the sketched residual norm, and the spectral term $\|\mat{\mathcal{L}}^{-1}\|_F$, which coincides with the reciprocal of the separation between the matrices
$\fullhessenberg{G}_d$ and $\fullhessenberg{H}_d$.
{\color{black}Indeed, starting from (\ref{eq:skecthedres_expression}) with $\fullcompressedsol{Y}_d^{\textsc{sk}}$ solving the reduced equation, the residual matrix can be written in the (ideal) orthonormal bases as
 { 
 \begin{align}\label{eq:skecthedres_expression}
    \mat{R}_d^\textsc{sk}
    ={\fullbasis{U}}_{d+1}\begin{bmatrix}
     0 & \fullcompressedsol{Y}_d^\textsc{sk}E_{d}{\widehat{\blockvec{R}}}_G^\top\\
     {\widehat{\blockvec{R}}}_H\blockvec{E}_d^{\top}\fullcompressedsol{Y}^\textsc{sk}_d & 0
 \end{bmatrix}{\fullbasis{V}}_{d+1}^{\top}, \nonumber
\end{align}}%
so that both addends in the bound factor $\|{\widehat{\blockvec{R}}}_H\blockvec{E}_d^{\top} \fullcompressedsol{Y}_d^{\textsc{sk}} + \fullcompressedsol{Y}_d^{\textsc{sk}}\blockvec{E}_d {\widehat{\blockvec{R}}}_G\phantom{)\!\!\!\!\!\!}^{\top}\|_F $ go to zero as the residual norm goes to zero.
}
Hence, for a well-posed full orthogonalization procedure, we expect {\color{black} that, in case of convergence,} the whitened-sketched scheme yields a solution not far from that of the full orthogonalization method.}

One could also write an analogous bound for the error norms in terms of the bottom entries of $\fullcompressedsol{Y}_d^{\textsc{full}}$, for which we have a better understanding of the convergence behavior. The proof follows the derivation above, by first substituting $\fullcompressedsol{Y}_d^{\textsc{sk}}$ with $(\fullcompressedsol{Y}_d^{\textsc{sk}}-\fullcompressedsol{Y}_d^{\textsc{full}})+ \fullcompressedsol{Y}_d^{\textsc{full}}$ into~(\ref{eq:compressedproblem_sketched_full}) to get 
\begin{align*}
&\hspace{-2cm}(\fullhessenberg{H}_d - {\widehat{\blockvec{R}}}_HE_d^\top) (\fullcompressedsol{Y}_d^{\textsc{full}}-\fullcompressedsol{Y}_d^{\textsc{sk}}) + 
(\fullcompressedsol{Y}_d^{\textsc{full}}-\fullcompressedsol{Y}_d^{\textsc{sk}}) (\fullhessenberg{G}_d - {\widehat{\blockvec{R}}}_GE_d^\top)^\top \\
&=
-{\widehat{\blockvec{R}}}_H\blockvec{E}_d^{\top} \fullcompressedsol{Y}_d^{\textsc{full}} - \fullcompressedsol{Y}_d^{\textsc{full}}\blockvec{E}_d{\widehat{\blockvec{R}}}_G\phantom{)\!\!\!\!\!\!}^{\top},
\end{align*}
so that
$$
\|\fullcompressedsol{Y}_d^{\textsc{full}}-\fullcompressedsol{Y}_d^{\textsc{sk}}\|_F \le \|\mat{\mathcal{\widetilde L}}^{-1}\|_F\, 
 \|{\widehat{\blockvec{R}}}_H\blockvec{E}_d^{\top} \fullcompressedsol{Y}_d^{\textsc{full}} + \fullcompressedsol{Y}_d^{\textsc{full}}\blockvec{E}_d {\widehat{\blockvec{R}}}_G\phantom{)\!\!\!\!\!\!}^{\top}\|_F, 
$$
where $\widetilde{\mat{\mathcal{L}}} = (\fullhessenberg{G}_d- {\widehat{\blockvec{R}}}_GE_d^\top)\otimes \mat{I} + \mat{I} \otimes (\fullhessenberg{H}_d- {\widehat{\blockvec{R}}}_HE_d^\top)$.
On the other hand, the spectral term $\|\widetilde{\mat{\mathcal{L}}}^{-1}\|_F$ of the resulting bound is not easy to control as it depends on low-rank modifications to $\fullhessenberg{H}_d$ and $\fullhessenberg{G}_d$.

Following the same reasonings as above, very similar results can be derived for the truncated approach.
In this case, we can also borrow some of the results from the theory
of truncated algorithms for (vector) linear systems.
Indeed, using the vector form, we have that $x_d^{\textsc{tr}}=\vecop(\mat{X}_d^{\textsc{tr}})=
(\truncbasis{V}_d\otimes \truncbasis{U}_d)y_d^{\textsc{tr}}$ with $y_d^{\textsc{tr}}=\vecop(\mat{Y}_d^{\textsc{tr}})$. Under this form we can
derive a relation with the optimal Arnoldi solution $x_d^{\textsc{full}}$, that is the one that keeps both bases orthonormal, namely $\fullbasis{V}_d\otimes \fullbasis{U}_d$. To this end, let $\mat{\mathcal{M}}=\mat{I}\otimes \mat{A} + \mat{B}\otimes \mat{I}$ and $\mat{\mathcal{W}}:=\truncbasis{V}_d\otimes \truncbasis{U}_d$ be the Kronecker sum and product with quantities as in \eqref{eqn:arnoldi_tr}, and let  $\mat{\mathcal{P}}_{\mat{\mathcal{M}}}$ be the oblique projection $\mat{\mathcal{P}}_{\mat{\mathcal{M}}}=\mat{\mathcal{M}}\mat{\mathcal{W}} (\mat{\mathcal{W}}^{\top}\!\mat{\mathcal{M}}\mat{\mathcal{W}})^{-1}\mat{\mathcal{W}}^{\top}$, where we assume that the inner matrix is nonsingular, that is, the two spaces spanned by $\truncbasis{U}_d, \truncbasis{V}_d$ keep growing. The following equality was proved in~\cite[Th.~3.1]{Simoncini2005},
$$
\frac{\|r_d^{\textsc{full}}-r_d^{\textsc{tr}}\|}{\|r_d^{\textsc{full}}\|} = 
\frac{\|\mat{\mathcal{P}}_{\mat{\mathcal{M}}}w\|}{\|(\mat{I}-\mat{\mathcal{P}}_{\mat{\mathcal{M}}})w\|}, \qquad
w=v_{d+1}\otimes u_{d+1},
$$
where $r_d^{\textsc{full}}, r_d^{\textsc{tr}}$ are the vector residuals corresponding to the above approximations {and for simplicity we have assumed a rank-one right-hand side matrix $C_1 C_2^\top$}.
This {\it equality} shows that as long as the next vectors of the bases, 
$u_{d+1}, v_{d+1}$, provide significant new information with respect to the previously computed 
vectors, then the two residuals will be close to each other. 
In other words, orthogonality is not paramount, only the angle between the new vectors and the previous spaces is.

\section{Theoretical investigation}\label{sec:theory}
In this section we explore some of the theoretical aspects related to the presented sketching technique. 

\subsection{On the field of values of the projected problem}

It is well-known that for standard Krylov subspace methods for Sylvester equations, the well-posedness of the projected problems can be guaranteed by assuming that the fields of values of $\mat{A}$ and $-\mat{B}$ have empty intersection, namely $W(\mat{A})\cap W(-\mat{B})=\emptyset$; see, e.g.,~\cite{Simoncini2016}. In the sketching framework, stronger assumptions have to be considered to ensure that equation~\eqref{eqn:reduced_sk} has a unique solution for any $d$.

In general, the field of values of the matrix associated with the sketching procedure is not a significant tool to analyze well-posedness, as sketching  may substantially change the field of values of $\mat{A}$. In the following we illustrate the quality of this change. To this end, we first describe a general result on ``sketched'' fields of values that will then be adapted to our setting.

 \begin{proposition} Let $\mat{S}$ be an $\varepsilon$-subspace embedding for $\mathcal{V}$. 
Then for any complex vector $v$ such that  $v, \mat{A}v \in \mathcal{V}$ it holds that
$$
|\Re(v^*\mat{S}^\top\!\mat{S A} v) - \Re(v^*\!\mat{A} v) | \le \varepsilon \|v\| \|\mat{A} v\|,
$$
{\color{black} with high probability,}
{\color{black} where $\Re(\cdot)$ denotes the real part of a complex number}, so that
\begin{equation*}\label{eq:locationW_SA}
\frac{1}{1+\varepsilon} \Re\left(\frac{v^*\!\mat{A} v}{v^*v}\right) - \varepsilon\|v\| \|\mat{A} v\| \le
\Re(v^* \mat{S}^\top\!\mat{S A} v) \le \frac{1}{1-\varepsilon} \Re\left(\frac{v^*\!\mat{A} v}{v^*v}\right) + \varepsilon\|v\|\|\mat{A} v\| .
\end{equation*}
\end{proposition}
\begin{proof}
Let $v = v_1 + \imath v_2$. Then
\begin{eqnarray*}
\Re(v^* \mat{S}^\top\!\mat{S A} v) 
&=&
\langle \mat{S} v_1, \mat{S} \mat{A} v_1\rangle + \langle \mat{S} v_2, \mat{S} \mat{A} v_2\rangle\\
&=& \left\langle (I_2\otimes \mat{S})\begin{bmatrix}v_1\\ v_2\end{bmatrix}, (I_2\otimes \mat{S})\begin{bmatrix}Av_1\\ Av_2\end{bmatrix}\right\rangle. 
\end{eqnarray*}
By \Crefrange{pro:SV_embedding}{pro:kron_embedding}, $I_2\otimes \mat{S}$ is an $\varepsilon$-subspace embedding for
\[
\mathcal{V}_2 := \left\{\begin{bmatrix}x\\ y\end{bmatrix} : x, y \in \mathcal{V}\right\},
\]
so that by~\eqref{eq:sketch_innerproduct},
\[
\Re(v^*\!\mat{A} v) - \varepsilon \|v \| \,  \|\mat{A} v \| \leq \Re(v^* \mat{S}^\top\!\mat{S A} v) \leq \Re(v^*\!\mat{A} v) + \varepsilon \|v \| \,  \|\mat{A} v \|.
\]
The factor involving $1\pm\varepsilon$ stems from (\ref{eq:sketch}).
\end{proof}

\!\!The bounds above help us appreciate the
distance between elements in the sketched and nonsketched field of values.
Indeed, let range$(\mat{\widehat U})=\mathcal{V}$. If we write $v= \mat{\widehat U}_dy$, with $y\ne 0$, then it holds that $\|\mat{S}_{\truncbasis{U}}v\|=\|y\|$ thanks to $\mat{S}_{\truncbasis{U}}^{\top}\!\mat{S}_{\truncbasis{U}}$-orthogonality of $\mat{\widehat U}_d$. Moreover, using
\Cref{prop:Sarnoldi}
    $$v^*\mat{S}_{\truncbasis{U}}^{\top}\!\mat{S}_{\truncbasis{U}}\mat{A} v= y^*\mat{\widehat U}_d^{\top}\mat{S}_{\truncbasis{U}}^{\top}\!\mat{S}_{\truncbasis{U}}\mat{A} \mat{\widehat U}_dy = y^*(\widehat{\trunchessenberg{H}}_d + \widehat{\blockvec{H}}\blockvec{E}_d^{\top})y.$$

    Hence, points of $W(\widehat{\trunchessenberg{H}}_d + \widehat{\blockvec{H}}\blockvec{E}_d^{\top})$ may be distant from {$W(\widehat{\truncbasis{U}}^{\top}\!\!\mat{A}\widehat{\truncbasis{U}})$} as much as $\mathcal{O}(\|\mat{A}\|)$. The following example makes this point clearer.

\begin{figure}
\centering
\includegraphics[width=3in,height=2.3in]{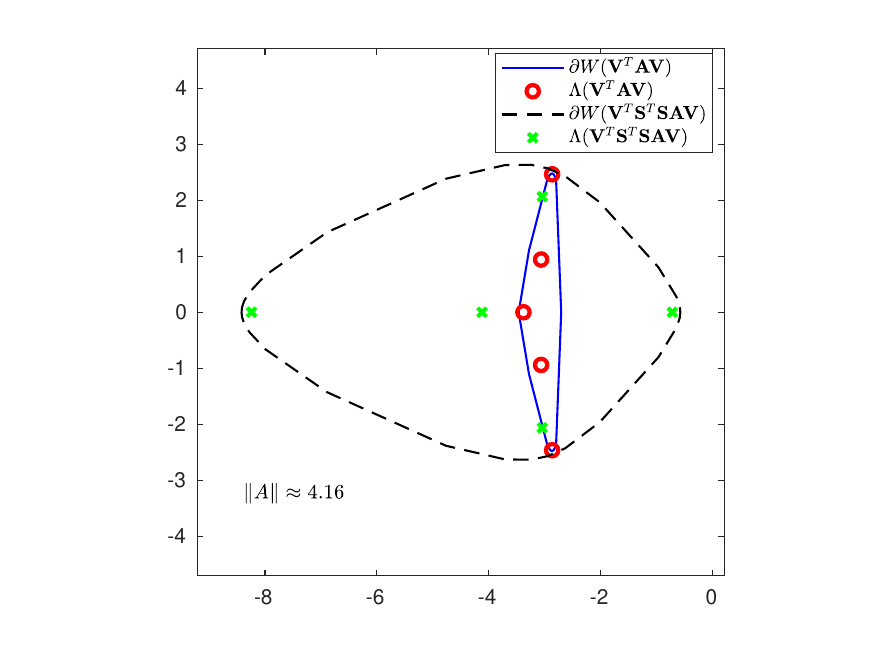}
\caption{\Cref{ex:fov1}. Fields of values $\partial W(\mat{V}^\top\!\!\mat{A}\mat{V})$ (solid blue line) and
$\partial W(\mat{V}^\top\!\mat{S}^\top\!\mat{S}\mat{A}\mat{V})$ (dashed black line), together with the corresponding eigenvalues.}\label{fig:fov1}
\end{figure}

\vskip 0.1in
\begin{example}\label{ex:fov1}
Consider the following code, building the Arnoldi basis matrix $\mat{V}$ for the given $\mat{A}$ of size $n\times n$;

\vskip 0.2in
\begin{verbatim}
n=30; rng(35); 
A=-toeplitz([3,-1,-0.5,zeros(1,n-3)],[3,1,1,zeros(1,n-3)]);
b=ones(n,1);b=b/norm(b);
K=[b,A*b,A*(A*b), A*(A*(A*b)), A*(A*(A*(A*b)))];  % Krylov matrix
d=size(K,2); s=2*d;
[V,R]=qr(K,0);  % Arnoldi basis
S = 1/sqrt(s)*randn(s,n);
\end{verbatim}
\vskip 0.2in

Note that the Arnoldi relation can be recovered as {\tt A*V(:,1:d-1) = V*H;} with {\tt H=R(1:d,2:d)/R(1:d-1,1:d-1);}.  
The sketching operator $\mat{S}$ is chosen as a random Gaussian matrix~\cite{oymak2018universality,MartinssonTropp2020}.
\Cref{fig:fov1} reports the fields of values\footnote{The boundary of the field of values is approximated by the function {\tt fv.m} in \cite{Higham.toolbox}.}  $W(\mat{V}^\top\!\!\mat{A}\mat{V})$ and
$W(\mat{V}^\top\!\mat{S}^\top\!\mat{S} \mat{A} \mat{V})$, together with the corresponding eigenvalues. The quantity
 $\|\mat{A}\|$ is also reported and the distance between the two boundaries clearly reaches the magnitude of this norm.
\end{example}
\vskip .1in

We next deepen the analysis of the properties of the projected equations~\eqref{eqn:reduced_sk}.
{\color{black}To simplify the presentation, we 
assume that the eigenvalues of $\mat{A}$ and $\mat{B}$  are} simple to avoid working with
invariant subspaces.

As a prerequisite for our analysis, we require the following technical lemma.
A similar result was proved in \cite[Section~7]{PalittaSchweitzerSimoncini2023} using the eigenvalue decomposition of the given matrix. To deal with matrix equations, here we generalize the result to the use of the Schur decomposition, which appears to be more appropriate in this context. {\color{black} We stress that this result, together with the subsequent \Cref{prop:reducedsol}, are possible thanks to the low-rank modification form for the projected matrix derived in \Cref{prop:Sarnoldi}. Expressions for the projected matrix using the pseudoinverse, such as in \cite{burke2023krylov,CortinovisKressnerNakatsukasa2022,GuettelSchweitzer2022}, would not allow one to obtain these simple derivations.}

\begin{lemma}\label{lemma:schur}
Let $(\widehat{\trunchessenberg{H}}_d+ \widehat h e_d^{\top})^{\top} =\schur{Q} \schur{R} \schur{Q}^*$ be the Schur decomposition of (the transpose of) the given matrix. For $\eta>0$, assume that there exists an index $\bar i\le d$ 
such that $\dist(\schur{R}_{ii},W(\widehat{\trunchessenberg{H}}_d))$ $=\mathcal{O}(\eta^{-1})$, $i=1, \ldots \bar i$ for an appropriate
spectral ordering in the Schur decomposition.
Then it holds that $\schur{Q}_{1,i}=\mathcal{O}((\eta/\xi)^{d})$, for $i=1, \ldots \bar i$, with 
$\xi=\max_{j\le \bar i} |1/(\widehat{\trunchessenberg{H}}_d)_{j+1,j}|$.
\end{lemma}

\begin{proof}
Let us write $\schur{Q}=[q_1, \ldots, q_d]$.
Consider the first column in the Schur decomposition, that is
$(\widehat{\trunchessenberg{H}}_d^{\top}+ e_d \widehat h^{\top}) q_1 = q_1 \schur{R}_{1,1}$. Let
$\mat{N}:=\widehat{\trunchessenberg{H}}_d^{\top}$ and recall that $\mat{N}$ is lower Hessenberg. 
For each row $k<d$ of this eigenvalue equation we have 
$$
\mat{N}_{k,1:k+1} q_{1:k+1,1} = q_{k,1} \schur{R}_{1,1}.
$$
Reordering terms, $\mat{N}_{k,k+1} q_{k+1,1} = (\mat{N}_{k,k} - \schur{R}_{1,1}) q_{k,1} + \mat{N}_{k,1:k-1} q_{1:k-1,1}$,
from which we obtain
$$
 q_{k+1,1} = \frac{(\mat{N}_{k,k} - \schur{R}_{1,1})}{ \mat{N}_{k,k+1}} q_{k,1} + 
\frac{\mat{N}_{k,1:k-1}}{ \mat{N}_{k,k+1}} q_{1:k-1,1} .
$$
Taking into account the unit norm of $q_1$ it follows
$$
| q_{k+1,1}| \le \frac{|\mat{N}_{k,k} - \schur{R}_{1,1}|}{|\mat{N}_{k,k+1}|} |q_{k,1}| + 
\frac{\|\mat{N}_{k,1:k-1}\|}{ |\mat{N}_{k,k+1}|}.
$$
Hence, the entries of $q_1$ satisfy the relation
$| q_{k+1,1}| = \mathcal{O}(\xi\eta^{-1} |q_{k,1}|)$. Normalization of $q_1$ yields the
sought after result.

For a subsequent Schur vector $q_i$ with $1<i\le \bar i$ we have
$$
\mat{N}_{k,1:k+1} q_{1:k+1,i} = q_{k,i} \schur{R}_{i,i} + q_{k,1:i-1}\schur{R}_{1:i-1,i},
$$
from which 
$\mat{N}_{k,k+1} q_{k+1,i} = (\mat{N}_{k,k} - \schur{R}_{i,i}) q_{k,i} + \mat{N}_{k,1:k-1} q_{1:k-1,i}+ q_{k,1:i-1}\schur{R}_{1:i-1,i}$
follows, so that
$$
 q_{k+1,i} = \frac{(\mat{N}_{k,k} - \schur{R}_{i,i})}{ \mat{N}_{k,k+1}} q_{k,i} + 
\frac{\mat{N}_{k,1:k-1}}{ \mat{N}_{k,k+1}} q_{1:k-1,i}
+ \frac{ \schur{R}_{1:i-1,i}}{ \mat{N}_{k,k+1}} q_{1:k-1,i-1}.
$$
Taking once again absolute values we finally obtain
$$
 |q_{k+1,i}| \le \frac{|\mat{N}_{k,k} - \schur{R}_{i,i}|}{|\mat{N}_{k,k+1}|} |q_{k,i}| + 
\frac{\|\mat{N}_{k,1:k-1}\|}{ |\mat{N}_{k,k+1}|} 
+ \frac{ \|\schur{R}_{1:i-1,i}\|}{ |\mat{N}_{k,k+1}|}, 
$$
so that $|q_{k+1,i}| = \mathcal{O}(\xi\eta^{-1} |q_{k,i}|)$.
\end{proof}

If some of the eigenvalues of $\widehat{\trunchessenberg{H}}_d+ \widehat h e_d^{\top}$
are significantly far from the field of values of $\widehat{\trunchessenberg{H}}_d$,
the first entry of the corresponding Schur vectors will be far below
machine precision,  orders of magnitude smaller, especially for $d\gg 10$.
Therefore, to simplify the derivation, in the following proposition we will assume
that these Schur vector entries are in  fact zero.
Such sparsity property
allows us to derive the following result on the rank structure of the solution to the
{\color{black}Sylvester equation.

 \begin{proposition}\label{prop:reducedsol}
{\it Let $\widehat{\trunchessenberg{H}}_d^{\top}+ e_d \widehat h^{\top}=\schur{Q} R \schur{Q}^*$ and $\widehat{\trunchessenberg{G}}_d^{\top}+ e_d \widehat g^{\top}=\schur{P} K \schur{P}^*$ be Schur decompositions of 
the given matrices, with $\schur{Q}=[q_1, \ldots, q_d]$ and $\schur{P}=[p_1, \ldots, p_d]$. 
Assume that the spectra of $\widehat{\trunchessenberg{H}}_d^{\top}+ e_d \widehat h^{\top}$ and 
$-(\widehat{\trunchessenberg{G}}_d^{\top}+ e_d \widehat g^{\top})$ do not  intersect. Under the hypotheses of
\Cref{lemma:schur} in the limit for $\eta=0$, define the partitions $\schur{Q}=[\schur{Q}_0, \schur{Q}_1]$ with $e_1^{\top} \schur{Q}_0=0$
and $e_1^{\top}\schur{Q}_1\ne 0$, and  $\schur{P}=[\schur{P}_0, \schur{P}_1]$ with $e_1^{\top} \schur{P}_0=0$
and $e_1^{\top}\schur{P}_1\ne 0$.
Then the solution $\mat{Y}_d$ to
\begin{equation}\label{eqn:orig}
(\widehat{\trunchessenberg{H}}_d+ \widehat h e_d^{\top}) \mat{Y}_d + \mat{Y}_d (\widehat{\trunchessenberg{G}}_d+ \widehat g e_d^{\top})^* = e_1 \beta_1\beta_2 e_1^*
\end{equation}
{\color{black} where $\beta_1, \beta_2 > 0$ are arbitrary}, is given by $\mat{Y}_d=\schur{Q}_1 \mathcal{Z} \schur{P}_1^*$ where $\mathcal{Z}$ solves
\begin{equation}\label{eqn:orig3}
\schur{Q}_1^*(\widehat{\trunchessenberg{H}}_d+ \widehat h e_d^{\top}) \schur{Q}_1 \mathcal{Z}  + \mathcal{Z} \schur{P}_1^* (\widehat{\trunchessenberg{G}}_d+ \widehat g e_d^{\top})^* \schur{P}_1 =  
\schur{Q}_1^* e_1\beta_1\beta_2  e_1^* \schur{P}_1.
\end{equation}
}
\end{proposition}

\begin{proof}
Spectral separation ensures existence and uniqueness of $\mat{Y}_d$.
Left and right multiplication of (\ref{eqn:orig}) by $\schur{Q}^*$ and $\schur{P}$, respectively,
gives
\begin{equation}\label{eqn:orig1}
\schur{Q}^*(\widehat{\trunchessenberg{H}}_d+ \widehat h e_d^{\top}) \schur{Q} \schur{Q}^* \mat{Y}_d \schur{P} 
+ \schur{Q}^* \mat{Y}_d \schur{Q} \schur{Q}^*(\widehat{\trunchessenberg{G}}_d+ \widehat g e_d^{\top})^* \schur{P} =
\schur{Q}^* e_1\beta_1 \beta_2^* e_1^* \schur{P}.
\end{equation}
Using the hypotheses, $\schur{Q}^* e_1\beta_1 = [0;q]$ with $q=\schur{Q}_1^* e_1\beta_1$, and
$\schur{P}^* e_1\beta_2 = [0;p]$ with $p=\schur{P}_1^* e_1\beta_2$.
Let us use the partitions
\begin{align*}
\schur{Q}^*(\widehat{\trunchessenberg{H}}_d+ \widehat h e_d^{\top}) \schur{Q}
 &= : \!
\begin{bmatrix} 
\schur{R}_{11}^* & 0\\
\schur{R}_{12}^* & \schur{R}_{22}^*
\end{bmatrix},
\,\,
&\schur{P}^*(\widehat{\trunchessenberg{G}}_d+ \widehat g e_d^{\top}) \schur{P}
 &= : \!
\begin{bmatrix} 
\schur{K}_{11}^* & 0\\
\schur{K}_{12}^* & \schur{K}_{22}^*
\end{bmatrix},
\\
 \schur{Q}^* \mat{Y}_d \schur{P} &= :\!
\begin{bmatrix} 
\schur{Z}_{11} & \schur{Z}_{12}\\
\schur{Z}_{21} & \schur{Z}_{22}
\end{bmatrix},
\,\, 
&\schur{Q}^* e_1 e_1^* \schur{P} &= :\!
\begin{bmatrix} 
0 & 0\\
0 &qp^* 
\end{bmatrix}.
\end{align*}
Then equating all blocks in (\ref{eqn:orig1}) we obtain
\begin{equation}\label{eqn:block11}
\schur{R}_{11}^* \schur{Z}_{11} + \schur{Z}_{11} \schur{K}_{11}=0,
\quad
\schur{R}_{12}^* \schur{Z}_{11} + \schur{R}_{22}^* \schur{Z}_{21} + \schur{Z}_{21} \schur{K}_{11} = 0,
\end{equation}
and
\begin{equation}\label{eqn:block12}
\schur{R}_{11}^* \schur{Z}_{12} + \schur{Z}_{11} \schur{K}_{12} + \schur{Z}_{12}\schur{K}_{22}=0,
\quad
 \schur{R}_{12}^* \schur{Z}_{12}+\schur{R}_{22}^* \schur{Z}_{22} + \schur{Z}_{21}\schur{K}_{12}+  \schur{Z}_{22} \schur{K}_{22}=qp^*.
\end{equation}
From the left matrix equation in (\ref{eqn:block11})
 we obtain $\schur{Z}_{11}=0$, since the spectra of $\schur{R}_{11}$ and $-\schur{K}_{11}$ do not intersect. 
This in turn leads to $\schur{Z}_{21}=0$ from the right equation, since $\schur{K}_{11}$ and $-\schur{R}_{22}$
have distinct eigenvalues.
Similarly, $\schur{Z}_{12}=0$ can be obtained from the left matrix equation in \eqref{eqn:block12}. The right matrix equation in (\ref{eqn:block12}) gives the
sought after solution $\schur{Z}=\schur{Z}_{22}$.
\end{proof}
}

In light of the previous  result, for $\schur{Q}_1\in {\mathbb C}^{d\times \ell}$,
we call the
{\it effective field of values} of $\widehat{\trunchessenberg{H}}_d+ \widehat h e_d^{\top}$ the set
$$
W_{\textnormal{eff}}(\widehat{\trunchessenberg{H}}_d + \widehat h e_d^{\top})= 
W(\schur{Q}_1^*(\widehat{\trunchessenberg{H}}_d + \widehat h e_d^{\top}) \schur{Q}_1)=
\left \{ \frac{ v^*\schur{Q}_1^*(\widehat{\trunchessenberg{H}}_d + \widehat h e_d^{\top}) \schur{Q}_1 v}{v^*v}, 
\,\, 0\ne v\in{\mathbb C}^{\ell}\right \}.
$$
{\color{black}The difference between the field of values and its effective counterpart depends on how large the distance defined in \Cref{lemma:schur} is, and it may be significant if for instance the rank-one modification remarkably changes the eigenvalues.}
The following example illustrates the difference between the field of values and the effective field of values.

\begin{figure}
\centering
\includegraphics[trim= 150 200 100 200,clip,scale=0.32]{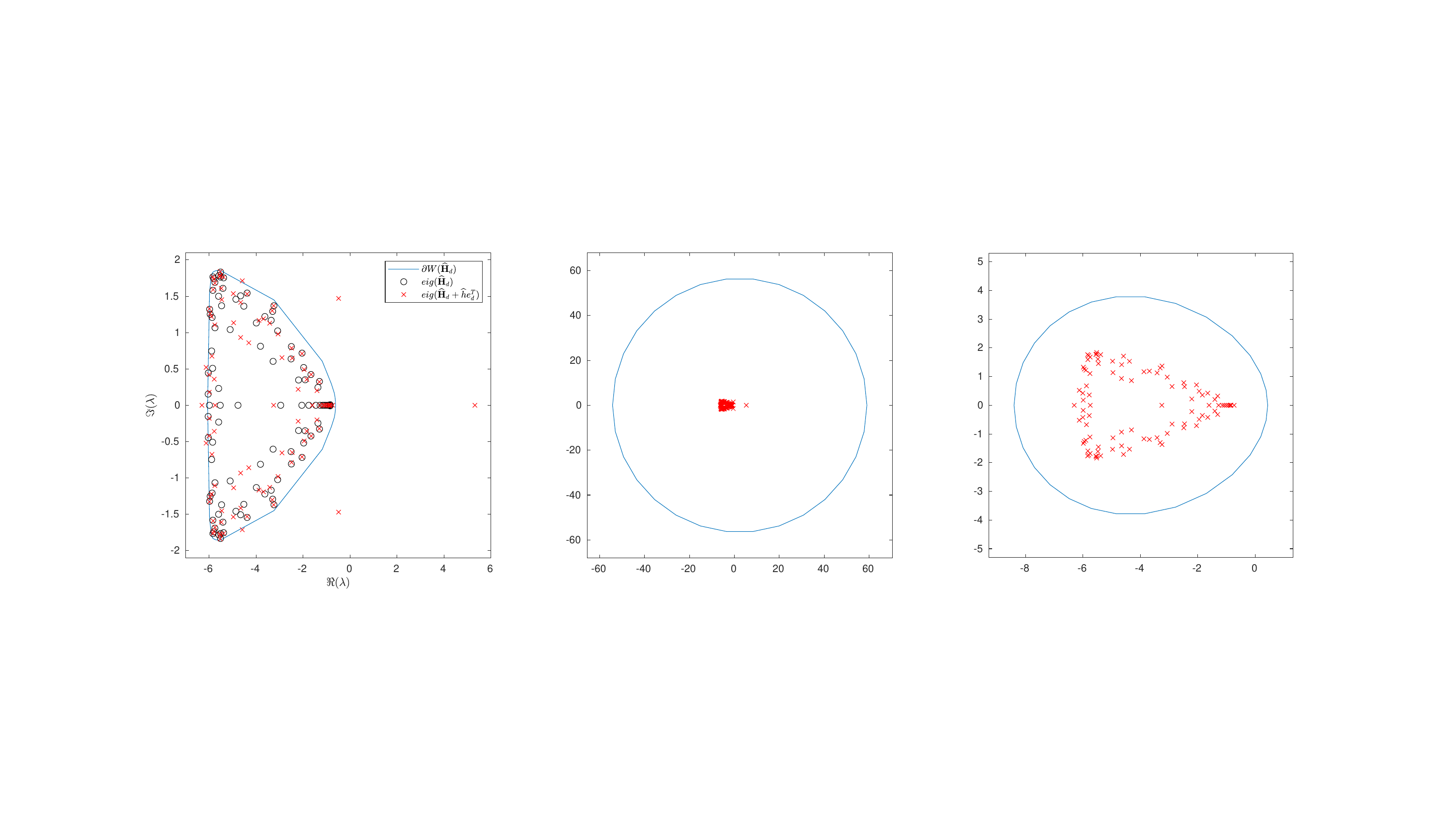}

\caption{\Cref{first_example}. Left: $\partial W(\mat{\widehat H}_d)$ (blue solid line), spectra of $\mat{\widehat H}_d$ (black circles), and $\mat{\widehat H}_d+\widehat he_d^\top$ (red crosses). Center: $\partial W(\mat{\widehat H}_d+\widehat he_d^\top)$ (blue solid line) and spectrum of $\mat{\widehat H}_d+\widehat he_d^\top$ (red crosses). Right: $\partial W(\schur{Q}_1^*(\mat{\widehat H}_d+\widehat he_d^\top)\schur{Q}_1)$ (blue solid line) and spectrum of $\schur{Q}_1^*(\mat{\widehat H}_d+\widehat he_d^\top)\schur{Q}_1$ (red crosses).} \label{fig_example4_5}
\end{figure}

\begin{example}\label{first_example}
For this example we take inspiration from~\cite[Example 5.6]{PalittaSchweitzerSimoncini2023}, {\color{black} and we set $\mat{B}=\mat{A}^\top$. Moreover}, we set $d=100$ and construct the matrix $\mat{\widehat H}_d$ and the vector $\widehat h$ as follows
\begin{quotation}
{\footnotesize
\noindent
\begin{verbatim}
rng(21*pi); d=100;
Hhat=toeplitz([-4, 2, zeros(1,d-2)], [-4, 1/2, 1/2, randn(1,d-3)/20]);
hhat=flipud(linspace(1,20,d))'.*randn(d,1);
\end{verbatim}
}
\end{quotation}
\vskip 0.1in

In \Cref{fig_example4_5} (left) we report $\partial W(\mat{\widehat H}_d)$ (blue solid line) along with the spectra of $\mat{\widehat H}_d$ (black circles) and $\mat{\widehat H}_d+\widehat he_d^\top$ (red crosses). This figure shows that, although $\mat{\widehat H}_d$ is negative definite, the matrix $\mat{\widehat H}_d+\widehat he_d^\top$ has eigenvalues with positive real part. In particular, the set $W(\mat{\widehat H}_d+\widehat he_d^\top)$ reported in \Cref{fig_example4_5} (center) significantly trespasses the imaginary axis. This may jeopardize the well-posedness of the projected equation~\eqref{eqn:orig} and/or the positive semidefiniteness of its solution $\mat{Y}_d$. This situation is mitigated by looking at the effective field of values of $\mat{\widehat H}_d+\widehat he_d^\top$, namely $W(\schur{Q}_1^*(\mat{\widehat H}_d+\widehat he_d^\top)\schur{Q}_1)$; see \Cref{fig_example4_5} (right). 
Indeed, $W(\schur{Q}_1^*(\mat{\widehat H}_d+\widehat he_d^\top)\schur{Q}_1)$ is much smaller than 
$W(\mat{\widehat H}_d+\widehat he_d^\top)$ and all the eigenvalues of $\schur{Q}_1^*(\mat{\widehat H}_d+\widehat he_d^\top)\schur{Q}_1$ have negative real part.
This means that the matrix $\mat{Y}_d$ can be constructed by first computing  the solution $\schur{Z}$ to~\eqref{eqn:orig3} and then setting $\mat{Y}_d=\schur{Q}_1\schur{Z}\schur{Q}_1^*$. This guarantees both the existence and the positive semidefiniteness of  $\mat{Y}_d$.
\end{example}

\subsection{Convergence analysis for the Lyapunov setting}

{\color{black} For the sake of simplifying the derived bounds, in this section we} focus  on the Lyapunov equation
\begin{equation}\label{eq:lyapunov}
\mat{A} \mat{X}+\mat{X}\mat{A}^{\top}=cc^{\top},
\end{equation}
where, for convenience, we assume that the right-hand side has rank~1 and that $\|c\|=1$. 

We next show that if $\mat{A}$ and the reduced matrix $\widehat{\trunchessenberg{H}}_d + \widehat{h}e_d^{\top}$ are both negative definite, a convergence analysis is possible in a very similar way to the standard, non-sketched setting. We closely follow the techniques from~\cite{simoncini2009convergence}.

\begin{theorem}\label{thm:lyapunov_faber}
Let $\mat{A}$ be negative definite and let $\mat{X}_d^\textsc{\textsc{sk}}$ denote the sketched-and-whitened Arnoldi approximation for the solution of~\eqref{eq:lyapunov}. Assume that $\widehat{\trunchessenberg{H}}_d + \widehat h e_d^{\top}$ is negative definite as well and denote $\alpha_{\max} := \max\{\alpha(\mat{A}),\alpha(\widehat{\trunchessenberg{H}}_d+\widehat h e_d^{\top})\}$.

Further, let $\mathbb{E} \subset \C^-$ be a compact, convex set containing the field of values of both $\mat{A}$ and $\widehat{\trunchessenberg{H}}_d + \widehat h e_d^{\top}$, let $\Phi_k, k = 0,1,\dots$ with $\deg(\Phi_k) = k$ denote the Faber polynomials with respect to $\mathbb{E}$ and let
\begin{equation}\label{eq:faber_series}
\exp(\tau z) = \sum\limits_{k=0}^\infty f_k(\tau) \Phi_k(z),
\end{equation}
be the Faber series of the exponential. Then,
{\color{black} with high probability,}
\begin{equation}\label{eq:convergence_faber_sketched}
\|\mat{X}-\mat{X}_d^{\textsc{\textsc{sk}}}\| \leq 2
\eta_{\varepsilon}
\sum_{k = d}^\infty \int_0^\infty \exp(\tau\alpha_{\max}) |f_k(\tau)| d\tau,
\end{equation}
with $\eta_{\varepsilon}=
\left(1+\sqrt{\frac{1+\varepsilon}{1-\varepsilon}}\right)\left(1+\frac{1}{\sqrt{1-\varepsilon}}\right)$ {\color{black} and $\varepsilon$ the parameter controlling the embedding quality; see~\eqref{eq:sketch} or~\eqref{eq:sketch_innerproduct}}.
\end{theorem}
\begin{proof}
We have $\mat{X}_d^\textsc{\textsc{sk}} = \widehat{\truncbasis{U}}_d\mat{Y}_d^{\textsc{sk}}\widehat{\truncbasis{U}}_d^{\top}$ with 
\begin{equation*}\label{eq:reduced_sk_mod}
(\widehat{\trunchessenberg{H}}_d + \widehat h e_d^{\top})\mat{Y}_d^{\textsc{sk}} + \mat{Y}_d^{\textsc{sk}} (\widehat{\trunchessenberg{H}}_d + \widehat h  e_d^{\top})^{\top} = e_1\beta^2e_1^{\top},
\end{equation*}
where $\beta = \|\mat{S}_{\truncbasis{U}}c\|$. Therefore, using the standard integral representation for the solution of the Lyapunov equation,
\begin{align}
\mat{X}_d^{\textsc{\textsc{sk}}} &= \widehat{\truncbasis{U}}_d\mat{Y}_d^{\textsc{sk}}\widehat{\truncbasis{U}}_d^{\top} \nonumber\\
 &= \int_0^\infty \widehat{\truncbasis{U}}_d\exp(\tau( \widehat{\trunchessenberg{H}}_d + \widehat h e_d^{\top}))e_1\beta^2e_1^{\top} \exp(\tau(\widehat{\trunchessenberg{H}}_d +\widehat  te_d^{\top})^{\top})\widehat{\truncbasis{U}}_d^{\top} d\tau, \nonumber\\ 
&= \int_0^\infty x^\textsc{\textsc{sk}}(\tau)(x^\textsc{\textsc{sk}}(\tau))^{\top} d\tau, \label{eq:integral_sketched}
\end{align}
with the short-hand notation $x^\textsc{\textsc{sk}}(\tau) := \widehat{\truncbasis{U}}_d\exp(\tau(\widehat{\trunchessenberg{H}}_d + \widehat h e_d^{\top}))e_1\beta$ for the sketched Krylov approximation of $x(\tau) := \exp(\tau \mat{A})c$.

Subtracting~\eqref{eq:integral_sketched} from the integral representation $\mat{X} = \int_0^\infty x(\tau)(x(\tau))^{\top} d\tau$ gives
\begin{align}
\|\mat{X}-\mat{X}_d^{\textsc{\textsc{sk}}}\| &\leq \int_0^\infty \|x(\tau)(x(\tau))^{\top}-x_d^\textsc{\textsc{sk}}(\tau)(x_d^\textsc{\textsc{sk}}(\tau))^{\top}\| \nonumber\\
&\leq \int_0^\infty (\|x(\tau)\| + \|x_d^{\textsc{\textsc{sk}}}(\tau)\|) \| x(\tau)-x_d^\textsc{\textsc{sk}}(\tau)\| d\tau \nonumber\\ 
&\leq (1+\|\widehat{\truncbasis{U}}_d\|)\int_0^\infty \exp(\tau\alpha_{\max}) \| x(\tau)-x^\textsc{\textsc{sk}}(\tau)\| d\tau , \label{eq:integral_error}
\end{align}
where for the last inequality we have used $\|x(\tau)\| \leq \exp(\tau\alpha_{\max})$  and $\|x_d^{\textsc{\textsc{sk}}}(\tau)\| \leq \exp(\tau\alpha_{\max})\|\widehat{\truncbasis{U}}_d\|$. To bound the norm inside the integral in~\eqref{eq:integral_error}, we use the Faber series~\eqref{eq:faber_series} together with the polynomial exactness property 
\begin{equation}\label{eq:polynomial_exactness}
\widehat{\truncbasis{U}}_d p_{d-1}(\widehat{\trunchessenberg{H}}_d + \widehat h e_d^{\top})e_1 = p_{d-1}(\mat{A})c,
\end{equation}
for any polynomial of degree at most $d-1$, which gives
\begin{equation}\label{eq:faber_difference}
x(\tau) - x_d^{\textsc{\textsc{sk}}}(\tau) = \sum_{k = d}^\infty f_k(\tau) (\Phi_k(\mat{A})c - \widehat{\truncbasis{U}}_d\Phi_k(\widehat{\trunchessenberg{H}}_d + \widehat h e_d^{\top})e_1\beta).
\end{equation}
Taking norms in~\eqref{eq:faber_difference} yields
\begin{equation}\label{eq:faber_difference_norm}
\|x(\tau) - x_d^{\textsc{\textsc{sk}}}(\tau)\| \leq \sum_{k = d}^\infty |f_k(\tau)| (\|\Phi_k(\mat{A})\| + \beta\|\widehat{\truncbasis{U}}_d\|\|\Phi_k(\widehat{\trunchessenberg{H}}_d + \widehat h e_d^{\top})\| ).
\end{equation}
By~\cite[Th\'eor\`eme~1.1]{beckermann2005image}, we have $\|\Phi_k(\mat{A})\| \leq 2, \|\Phi_k(\widehat{\trunchessenberg{H}}_d + \widehat h e_d^{\top})\| \leq 2$, by~\cite[Corollary~2.2]{balabanov2022randomized} it holds $\|\widehat{\truncbasis{U}}_d\| \leq \frac{1}{\sqrt{1-\varepsilon}}$ and further $\beta = \|\mat{S}_{\truncbasis{U}}c\|\leq \sqrt{1+\varepsilon}\|c\| = \sqrt{1+\varepsilon}$, because $\mat{S}_{\truncbasis{U}}$ is an $\varepsilon$-subspace embedding for $\spK_{d+1}(\mat{A},c).$ Inserting these relations into~\eqref{eq:faber_difference_norm} gives
\[
\|x(\tau) - x_d^{\textsc{\textsc{sk}}}(\tau)\| \leq 2\left(1+\sqrt{\frac{1+\varepsilon}{1-\varepsilon}}\right)\sum_{k = d}^\infty |f_k(\tau)|.
\]
Further inserting this into~\eqref{eq:integral_error} and using the bound $\|\widehat{\truncbasis{U}}_d\| \leq \frac{1}{\sqrt{1-\varepsilon}}$ once again completes the proof.
\end{proof}

Note that for the usually assumed sketching quality $\varepsilon = 1/\sqrt{2}$, we have
\[
2\eta_{\varepsilon}
\approx 19.45.
\]

The result of \Cref{thm:lyapunov_faber} is very similar to {\color{black} standard results for full Arnoldi in the (non-sketched) setting; see~\cite{simoncini2009convergence}}. In that case, a bound like~\eqref{eq:convergence_faber_sketched} holds, where the factor
$2\eta_{\varepsilon}$
is replaced by $8$, one can use $\exp(\tau\alpha(\mat{A}))$ instead of $\exp(\tau\alpha_{\max})$, and $\mathbb{E}$ only needs to contain $W(\mat{A})$, as it then also contains $W(\fullhessenberg{H}_d)$. This shows that as long as $W(\widehat{\trunchessenberg{H}}_d + \widehat{h}e_d^{\top})$ is not much larger than $W(\mat{A})$, one can expect the sketched method to converge similarly to the {\color{black} full Arnoldi method}. {\color{black} Let us also note that it is well-known that bounds of the form~\eqref{eq:convergence_faber_sketched} are often not tight in practice, also in the deterministic case. In particular, while the convergence slope might often be predicted quite well, the constants involved are often a large overestimate. These bounds are therefore not well-suited to be used as a stopping criterion for the method.}

{\color{black} We would like to remark that the polynomial exactness property~\eqref{eq:polynomial_exactness} used in the proof of \Cref{thm:lyapunov_faber} is readily available thanks to the WS-Arnoldi relations~\eqref{eqn:WSarnoldi}. Indeed, in~\cite[Theorem 2.7]{FroLS20} it has been shown that the only modifications to the upper Hessenberg matrix stemming from the Arnoldi process ($\widehat{\mathbf{H}}_d$ in our case) that preserve the polynomial exactness are indeed low-rank modifications acting only on the last (block) column. Once again, this result would not be available if we had to rely on expressions of the projected matrices involving pseudoinverses.}

For specific shapes of the set $\mathbb{E}$, the bound in \Cref{thm:lyapunov_faber} can be made more explicit. The following corollary exemplifies this for the case that $\mathbb{E}$ is an ellipse. {\color{black} It depends a lot on the specific $\mathbf{A}$ at hand whether assuming an enclosing ellipse is a reasonable assumption or not, and we merely show this as an example of the type of bound that can be obtained. For some other cases, e.g., a wedge-shaped set, similar bounds are possible.}

\begin{corollary}\label{cor:convergence_ellipse}
{\it Let the assumptions of \Cref{thm:lyapunov_faber} hold and further assume that $\mathbb{E}$ is an ellipse in $\mathbb{C}^-$ with center $(-c,0)$, foci $(-c,\pm \delta)$ and semi-axes $a_1$ and $a_2$, so that $\delta = \sqrt{a_1^2-a_2^2}$. Define $\rho_1 = \frac{a_1+a_2}{2}$ and $\rho_2 = \frac{c-\alpha_{\max}}{2\rho_1} + \frac{1}{2\rho_1}\sqrt{(c-\alpha_{\max})^2-\delta^2}$. Then
\[
\|\mat{X}-\mat{X}_d^{\textsc{\textsc{sk}}}\| \leq 
2\eta_{\varepsilon}
\frac{1}{\sqrt{(c-\alpha_{\max})^2-\delta^2}} \frac{\rho_2}{\rho_2-1}\left(\frac{1}{\rho_2}\right)^d.
\]
}
\end{corollary}
\begin{proof}
{\color{black}The proof works along the same lines as that of~\cite[Proposition~4.1]{simoncini2009convergence}.}
%
\end{proof}

The results above may appear to be of rather limited use, since  
$W(\widehat{\trunchessenberg{H}}_d + \widehat{h}e_d^{\top})$ may significantly
differ from $W(\mat{A})$. However, the results that led us to introduce the
effective field of values of $\widehat{\trunchessenberg{H}}_d + \widehat{h}e_d^{\top}$ enable us to obtain more informative convergence bounds.
Indeed, under the hypotheses of \Cref{prop:reducedsol}, we have
{\footnotesize
\begin{eqnarray*}
\mat{X}_d^{\textsc{\textsc{sk}}} 
&=& (\widehat{\truncbasis{U}}_d \schur{Q}_1) \schur{Z} (\widehat{\truncbasis{U}}_d \schur{Q}_1)^* \nonumber\\
&= &
 \int_0^\infty \widehat{\truncbasis{U}}_d \schur{Q}_1
\exp(\tau( \schur{Q}_1^*( \widehat{\trunchessenberg{H}}_d + \widehat h e_d^{\top})\schur{Q}_1)\schur{Q}_1^* e_1\beta^2e_1^{\top} \schur{Q}_1
\exp(\tau(\schur{Q}_1^*(\widehat{\trunchessenberg{H}}_d +\widehat  te_d^{\top})\schur{Q}_1)^*)\schur{Q}_1^*\widehat{\truncbasis{U}}_d^{\top} d\tau \nonumber\\ 
&=& \int_0^\infty x^\textsc{\textsc{sk}}(\tau)(x^\textsc{\textsc{sk}}(\tau))^{\top} d\tau, 
\end{eqnarray*}
}
where  
$$
x^\textsc{\textsc{sk}}(\tau)
=\widehat{\truncbasis{U}}_d \exp(\tau( \widehat{\trunchessenberg{H}}_d + \widehat h e_d^{\top})) e_1 \beta
=\widehat{\truncbasis{U}}_d \schur{Q}_1  \exp(\tau( \schur{Q}_1^*( \widehat{\trunchessenberg{H}}_d + \widehat h e_d^{\top})\schur{Q}_1))\schur{Q}_1^* e_1 \beta,
$$
and the last equality follows from noticing that under the proposition hypotheses,
$\schur{Q}^*e_1 = [0;\schur{Q}_1^* e_1]$.
In particular, in the proof of \Cref{thm:lyapunov_faber}, we have that
$\|x^\textsc{\textsc{sk}}(\tau)\| \le \exp(\tau \alpha_{\max}')$, where
$\alpha_{\max}'$ is the rightmost point of $W_{\textnormal{eff}}(\widehat{\trunchessenberg{H}}_d+ \widehat h e_d^{\top})$.
 Moreover,
\begin{equation*}\label{eq:faber_difference_norm1}
\|x(\tau) - x_d^{\textsc{\textsc{sk}}}(\tau)\| \leq 
\sum_{k = d}^\infty |f_k(\tau)| (\|\Phi_k(\mat{A})\| + 
\beta\|\widehat{\truncbasis{U}}_d\|\|\Phi_k(\schur{Q}_1^*(\widehat{\trunchessenberg{H}}_d + \widehat h e_d^{\top})\schur{Q}_1)\| ).
\end{equation*}

It thus follows that we can obtain a result such as \Cref{thm:lyapunov_faber} based on the \emph{effective} field of values of $\widehat{\trunchessenberg{H}}_d + \widehat h e_d^{\top}$, which can be significantly  closer to that of $\widehat{\trunchessenberg{H}}_d$  than $W(\widehat{\trunchessenberg{H}}_d + \widehat h e_d^{\top})$.

Corresponding convergence results can be obtained for the Sylvester equation, taking into account spectral properties of both left and right coefficient matrices, with only technical modifications. To avoid proliferation of similar results, we omit this derivation.

\section{The algorithms}\label{sec:algorithms}
We report the complete algorithms presented in the previous sections for solving the
Sylvester equation.
In \Cref{alg:whitening_krylov} we depict the sketched-and-truncated approach {\color{black}(with basis whitening)} for Sylvester equations. {\color{black}We label this algorithm \sPKSM.} The truncated Arnoldi method for~\eqref{eq:sylvester_equation} is illustrated in \Cref{alg:standardinnerprod_krylov} in \Cref{appendix:algorithm} and {\color{black}is labeled \tPKSM.}

Along with the definition of the projected problem and the residual norm expression, the main difference between {\color{black}\tPKSM} and {\color{black}\sPKSM} lies in the update of the
QR factorizations of the sketched bases in \cref{alg:QR_facts} of \Cref{alg:whitening_krylov}.
Similarly, the new block upper Hessenberg matrix
$\widehat{\trunchessenberg{H}}_{d+1} = \mat{T}_{\mat{U},d+1}\trunchessenberg{H}_{d+1}\mat{T}_{\mat{U},d+1}^{-1}$ is not
recomputed from scratch but rather updated: 
if $\mat{T}_{\mat{U},d+1}$ is as in \Cref{prop:sk_arn_rel}, then a direct computation shows that $\mat{T}_{\mat{U},d+1}^{-1}=\begin{bmatrix}
\mat{T}_{\mat{U},d}^{-1} & {\color{black}-}\mat{T}_{\mat{U},d}^{-1}\blockvec{T}_{H}\blockscalar{\tau}_{d+1}^{-1} \\
& \blockscalar{\tau}_{d+1}^{-1}\\
\end{bmatrix}$ so that

\begin{align*}
\small
\widehat{\trunchessenberg{H}}_{d+1}=
\begin{bmatrix}
   \widehat{\trunchessenberg{H}}_{d}+\blockvec{T}_{H}\blockscalar{h}_{d+1,d}{\blockscalar\tau}_{d}^{-1}\blockvec{E}_d^{\top} & \widehat{\blockvec{H}}_{\textnormal{new}}\\
{\blockscalar\tau}_{d+1}\blockscalar{h}_{d+1,d}{\blockscalar\tau}_{d}^{-1}\blockvec{E}_d^{\top} & 
{\blockscalar\tau}_{d+1}({\color{black}-}\blockscalar{h}_{d+1,d}{\blockscalar\tau}_{d}^{-1}\blockvec{E}_d^{\top}\blockvec{T}_{H}+
\blockscalar{h}_{d+1,d+1}){\blockscalar\tau}_{d+1}^{-1}\\
\end{bmatrix} ,
\end{align*}
where 
$\widehat{\blockvec{H}}_{\textnormal{new}}={\color{black}-}(\widehat{\trunchessenberg{H}}_{d}+
\blockvec{T}_{H}\blockscalar{h}_{d+1,d}{\blockscalar\tau}_{d}^{-1}\blockvec{E}_d^{\top})\blockvec{T}_{H}{\blockscalar\tau}_{d+1}^{-1}+
{\color{black}\mat{T}_{\truncbasis{U},d}}\blockvec{H}{\blockscalar\tau}_{d+1}^{-1}+\blockvec{T}_{H}\blockscalar{h}_{d+1,d+1}{\blockscalar\tau}_{d}^{-1}$
and $\blockvec{H}=[\blockscalar{h}_{1,d+1}^{\top},\ldots, \blockscalar{h}_{d,d+1}^{\top}]^{\top}\in\mathbb{R}^{dr\times r}$.
A corresponding update for $\widehat{\trunchessenberg{G}}_{d+1}$ is performed.

\begin{algorithm}[t]
\caption{Sketched-and-truncated Arnoldi method for Sylvester equations~(\sPKSM)}\label{alg:whitening_krylov}
\begin{algorithmic}[1]
\setstretch{1.2}
\smallskip

\Statex \textbf{Input:} \ \ $\mat{A}\in\R^{n\times n}$, $\mat{B}\in \R^{n \times n}$,  $\blockvec{C}_1 \in \R^{n\times r}$, $\blockvec{C}_2 \in \R^{n\times r}$, $\mat{S}_{\truncbasis{U}}\in\mathbb{R}^{s\times n}$, $\mat{S}_{\truncbasis{V}}\in\mathbb{R}^{s\times n_2}$,
integers $0<k\leq\text{maxit} \ll n$, $\tol>0$, $p\geq 1$ 
\Statex \textbf{Output:} $\mat{X}^{(1)},\mat{X}^{(2)}$ such that $\mat{X}^{(1)}(\mat{X}^{(2)})^{\top}=\mat{X}_d$ approximately solves $\mat{A} \mat{X}+\mat{X}\mat{B}=\blockvec{C}_1\blockvec{C}_2^{\top}$
\State Compute skinny QRs: $\blockvec{U}_1\blockscalar{\ell}=\blockvec{C}_1$, $\blockvec{V}_1\blockscalar{s}=\blockvec{C}_2$, $\mat{Q}_{\mat{U},1}\blockscalar{\beta}_1=\mat{S}_{\truncbasis{U}}\blockvec{C}_1$, $\mat{Q}_{V,1}\blockscalar{\beta}_2=\mat{S}_{\truncbasis{U}}\blockvec{C}_2$, and set $\mat{T}_{\mat{U},1}=\blockscalar{\beta}_1$, $\mat{T}_{V,1}=\blockscalar{\beta}_2$
 \For{$d=1,\ldots,\text{maxit}$}
\State Compute $\,\,\widetilde{\blockvec{U}}=\mat{A} \blockvec{U}_d$, $\,\,\widetilde{\blockvec{V}}=\mat{B}^{\top}\blockvec{V}_d$
\For{$i=\max\{1,d-k+1\},\ldots,d$}
\State Set $\widetilde{\blockvec{U}}= \widetilde{\blockvec{U}}-\blockvec{U}_i\blockscalar{h}_{i,d}\,\,$
with $\,\,\blockscalar{h}_{i,d}=\blockvec{U}_i^{\top}\widetilde{\blockvec{U}}$
\State Set $\widetilde{\blockvec{V}}= \widetilde{\blockvec{V}}-\blockvec{V}_i\blockscalar{g}_{i,d}\,\,$
with $\,\,\blockscalar{g}_{i,d}=\blockvec{V}_i^{\top}\widetilde{\blockvec{V}}$
\EndFor
\State Compute skinny QRs:  $\blockvec{U}_{d+1}\blockscalar{h}_{d+1,d}=\widetilde{\blockvec{U}}$ and $\blockvec{V}_{d+1}\blockscalar{g}_{d+1,d}=\widetilde{\blockvec{V}}$
\State Update QRs: $\mat{Q}_{\mat{U},d+1}\mat{T}_{\mat{U},d+1}=\mat{S}_{\truncbasis{U}}[\truncbasis{U}_{d},\blockvec{U}_{d+1}]$, 
\Statex\qquad\qquad\qquad\qquad$\mat{Q}_{V,d+1}\mat{T}_{V,d+1}=\mat{S}_{\truncbasis{V}}[\truncbasis{V}_d,\blockvec{V}_{d+1}]$~\label{alg:QR_facts}
\State Update $\widehat{\trunchessenberg{H}}_d=\mat{T}_{\mat{U},d}\trunchessenberg{H}_d\mat{T}_{\mat{U},d}^{-1}$, $\widehat{\blockvec{H}}=\blockvec{T}_{H}\blockscalar{h}_{d+1,d}\blockscalar{\tau_d}^{-1}$, 
\Statex\qquad\qquad\quad\ $\widehat{\trunchessenberg{G}}_d=\mat{T}_{V,d}\trunchessenberg{G}_d\mat{T}_{V,d}^{-1}$, $\widehat{\blockvec{G}}=\blockvec{T}_{G}\blockscalar{g}_{d+1,d}\blockscalar{\theta}_d^{-1}$ 
\If{$\mathtt{mod}(d,p)=0$}
\State Solve $\,\,(\widehat{\trunchessenberg{H}}_d+\widehat{\blockvec{H}}\blockvec{E}_d^{\top})  \mat{Y} +  \mat{Y} (\widehat{\trunchessenberg{G}}_d+\widehat{\blockvec{G}} \blockvec{E}_d^{\top})^{\top} = \blockvec{E}_1\blockscalar{\beta}_1\blockscalar{\beta}^{\top}_2\blockvec{E}_1^{\top}$ for $\mat{Y}$
\State Compute $\rho=\sqrt{\|\blockscalar{h}_{d+1,d}\blockvec{E}_d^{\top} \mat{Y}\|_F^2+\| \mat{Y}\blockvec{E}_d\blockscalar{g}_{d+1,d}^\top\|_F^2}$~\label{alg:line_rho_whitening}
\If{$\rho<\tol$}
\State Go to line~\ref{line:endfor_whitening}
\EndIf
\EndIf
\EndFor
\State Compute (possibly low-rank) factors $\mat{Y}_1, \mat{Y}_2$ such that 
$\mat{Y}\approx \mat{Y}_1\mat{Y}_2^{\top}$\label{line:endfor_whitening}
\State Retrieve $\mat{X}^{(1)}= \truncbasis{U}_d\mat{T}_{\mat{U},d}^{-1}\mat{Y}_1$, $\mat{X}^{(2)}=\truncbasis{V}_d\mat{T}_{V,d}^{-1} \mat{Y}_2$ by the two-pass step 
\label{line:end_sol}
\end{algorithmic}
\end{algorithm}

At convergence, the matrix $\mat{T}_{\mat{U},d}^{-1}\mat{Y}_d (\mat{T}_{V,d}^{-1})^{\top}$ is approximated by
a low-rank  factorization, 
$\mat{T}_{\mat{U},d}^{-1}\mat{Y}_d (\mat{T}_{V,d}^{-1})^{\top}
\approx \mat{Z}^{(1)}(\mat{Z}^{(2)})^{\top}$, with $\mat{Z}^{(1)},\mat{Z}^{(2)}\in\mathbb{R}^{dr\times \ell}$, $\ell\leq dr$ so that the final approximate 
solution $\mat{X}_d$ is available as the low-rank  product 
$\mat{X}_d\approx  \mat{X}^{(1)}(\mat{X}^{(2)})^{\top}=(\widehat{\truncbasis{U}}_d \mat{Z}^{(1)})(\widehat{\truncbasis{V}}_d \mat{Z}^{(2)})^{\top}$; cf.~\crefrange{line:endfor_whitening}{line:end_sol} of \Cref{alg:whitening_krylov}. 
 Since $\widehat{\truncbasis{U}}_d, \widehat{\truncbasis{V}}_d$ are not stored,
to finalize the computation a two-pass strategy is employed; see, e.g.~\cite{Kressner2008,PalittaSimoncini2018,FrommerSimoncini2008}.
More precisely, the original Arnoldi iteration is performed once again to iteratively
construct the two matrices $\widehat{\truncbasis{U}}_d\mat{Z}^{(1)}, \widehat{\truncbasis{V}}_d\mat{Z}^{(2)}$.
Since the orthogonalization coefficients are available in the matrices $\trunchessenberg{H}_d,  \trunchessenberg{G}_d$,
the cost of this second Arnoldi iteration does not involve inner products for
the local orthogonalization.
To improve efficiency, the reconstruction of $\widehat{\truncbasis{U}}_d\mat{Z}^{(1)}$ and 
$\widehat{\truncbasis{V}}_d\mat{Z}^{(2)}$
should be performed in chunks of $kr$ basis vectors at a time.

\Cref{tab_memory} summarizes the main memory requirements of full Arnoldi and the proposed methods. {\color{black}Both  \sPKSM} and {\color{black}\tPKSM} require the same amount of long vector allocations for fixed truncation parameter $k$ and number of iterations and we thus report only the storage demand of the former.
\begin{table}[t]
\centering
{\color{black}
\setlength{\tabcolsep}{3.1pt}
\begin{tabular}{|l|rr|rr|c|}
\hline
  &  \multicolumn{2}{c|}{Memory Allocation} & 
\multicolumn{2}{c|}{Memory Allocation} & \multicolumn{1}{c|}{Extra}\\
Method  &  \multicolumn{2}{c|}{Bases} & 
\multicolumn{2}{c|}{Approximate Solution} & \multicolumn{1}{c|}{Terms}\\
\hline
 Full Arnoldi       & $\fullbasis{U}_d, \fullbasis{V}_d$   & $2(d+1)r\cdot n$  & 
 $(\mat{X}_d^{\textsc{full}})^{(1)},
(\mat{X}_d^{\textsc{full}})^{(2)}$  & $2
\ell\cdot n$ & ${\mathcal O}(d^2r^2)$\\
\sPKSM &  
 $\widehat{\truncbasis{U}}_d, \widehat{\truncbasis{V}}_d$  & $2(k+1)r\cdot n$ & 
$\mat{X}^{(1)}, \mat{X}^{(2)}$& $2
\ell\cdot n$& ${\mathcal O}(d^2r^2+drs)$\\
\hline
  &  \multicolumn{2}{c|}{Cost} & 
\multicolumn{2}{c|}{Cost} & \multicolumn{1}{c|}{Cost}\\
&  \multicolumn{2}{c|}{Bases Construction} & 
\multicolumn{2}{c|}{Solution Projected Problem} & \multicolumn{1}{c|}{Two-Pass}\\
\hline
 Full Arnoldi       & \multicolumn{2}{c|}{$\mathcal{O}(\theta+nr^2d^2)$}& 
\multicolumn{2}{c|}{$\mathcal{O}(r^3d^3)$}  
 & -\\
\sPKSM &  
\multicolumn{2}{c|}{\color{black}$\mathcal{O}(\theta+nr^2dk+sr^2d^2)$}
&\multicolumn{2}{c|}{$\mathcal{O}(r^3d^3)$}  
& $\mathcal{O}(\theta+nr^2dk)$\\

\hline
\end{tabular}
\caption{Memory requirements and main computational costs (in flops) of full Arnoldi and \sPKSM. $\theta=\mathcal{O}((\nnz(A)+\nnz(B))rd)$.}\label{tab_memory}
}
\end{table}
For $k, \ell\ll d$ the main difference between the two classes of methods stands in the memory allocations for the generated bases. On the other hand, storage for the solution factors depends on the spectral decay properties of the solution to the reduced problem. Since the polynomial Krylov space generates redundant information, the numerical rank $\ell$ of $\mat{Y}_d$ may be significantly lower than its dimension\footnote{As already mentioned, we are assuming that the exact solution can be well approximated by a matrix of very low rank.}, allowing major savings in storing the approximate solution in factored form. We refer to the examples of \Cref{sec:numerical_examples} for experimental evidence. The extra allocations  concern matrices in the reduced spaces, together with the $s\times dr$ sketched basis\footnote{{\color{black}This bit is not present in the overall storage demand of \tPKSM.}}, which is kept orthonormal.

{\color{black}In \Cref{tab_memory} we also list the main computational costs of full Arnoldi and \sPKSM~in terms of flops. These amount to the bases construction and the solution of the projected problems. \sPKSM~also sees the extra cost coming from the second Arnoldi pass. The main difference between the two procedures lies in the orthogonalization cost. The full orthogonalization of Arnoldi requires $\mathcal{O}(nr^2d^2)$ flops whereas our sketched and truncated procedure cuts the cost down to $\mathcal{O}(nr^2dk+sr^2d^2)$ where the first term comes from the truncated Arnoldi step whereas the second one is related to the whitening of the sketched bases. The results in \Cref{sec:numerical_examples} will show that these savings in the number of flops do indeed lead to less time-consuming solution procedures. Notice that \tPKSM~and \sPKSM~basically have the same computational cost per iteration. The only difference is in the lack of any sketch-related cost in the former.}

{\color{black}Let us also mention that it is possible to use more stable means of orthogonalizing the bases in \sPKSM, e.g., a rank-revealing QR decomposition, as it was already suggested for whitening in~\cite{rokhlin2008fast} or a randomized Householder QR algorithm~\cite{grigori2023randomized}. This substantially increases the computational cost of the algorithm, though: Using a rank-revealing QR decomposition means that it is not possible to update the QR factors from one iteration to the next (cf.~\cref{alg:QR_facts} of \Cref{alg:whitening_krylov}). Instead, these factors have to be recomputed from scratch in every single iteration. Such an approach should therefore only be used for very challenging problems in which \sPKSM~becomes unstable. We note that we did not observe this situation in our experiments and therefore did not pursue such an approach any further.}

 
Finally, we remark that our stopping criterion is based on the residual norm or
its estimation, in both algorithms.
This criterion is checked periodically, not necessarily at each iteration,
to save computational efforts.

\section{Numerical experiments}\label{sec:numerical_examples}
In this section we report numerical experiments to illustrate the potential of the new algorithms. In all the experiments presented here, we employ the sketching matrix $\mat{S}=\sqrt{s/n}\cdot \mat{DNE}\in\mathbb{R}^{s\times n}$ where $\mat{E}\in\mathbb{R}^{n\times n}$ is a diagonal matrix containing Rademacher random variables on its diagonal, $\mat{N}\in\mathbb{R}^{n\times n}$ denotes the discrete cosine transform matrix, and $\mat{D}\in\mathbb{R}^{s\times n}$ randomly selects $s$ rows.\footnote{For reproducibility purposes, we set {\tt rng('default')} in all our experiments, and  the action of $\mat{N}$ is computed by the Matlab function {\tt dct}.}

All the experiments have been run using Matlab (version 2022b) on a machine with a 1.2GHz Intel quad-core i7 processor with 16GB RAM on an Ubuntu 20.04.2 LTS operating system. {\color{black} All the small-dimensional Sylvester equations arising in the projection-based procedures we test are solved by the Matlab function {\tt lyap} which implements the Bartels-Stewart algorithm~\cite{BartelsStewart1972}.

The code for reproducing the results presented in the following is publicly available online at 
\vspace{0.1cm}

{\small\url{https://github.com/palittaUniBO/Sketched-and-truncated-Krylov4MatrixEqs}}}

\begin{example}\label{Example 2}
We compare the new sketched-and-truncated method with the full Arnoldi method. Moreover, we analyze the role of the truncation parameter in the performance of the new algorithm.

We consider the matrices $\mat{A}$ and $\mat{B}$ in~\eqref{eq:sylvester_equation} coming from the discretization of a convection-diffusion differential operator of the form
\begin{equation}\label{eq:condiff_op}
    \mathcal{L}(u)=-\nu\Delta u+w\nabla u,
\end{equation}
on the unit square $[0,1]^2$ by centered finite differences using {\color{black}$\sqrt{n}$} nodes in each direction, yielding $\mat{A}, \mat{B} \in \mathbb{R}^{n \times n}$.
We consider the same viscosity parameter $\nu$ in the construction of both $\mat{A}$ and $\mat{B}$ whereas for the convection vectors we use $w_{\mat{A}}=(1,1)$ for $\mat{A}$ and $w_{\mat{B}}=(3y(1-x^2),-2x(1-y^2))$ for $\mat{B}$. {\color{black}The block vectors $C_1,C_2\in\mathbb{R}^{n\times r}$ in~\eqref{eq:sylvester_equation} have random entries drawn from the standard normal distribution and are such that $\|C_1C_2^{\top}\|_F=1$}.

\begin{figure}
\centering
\includegraphics[width=5in,height=2.3in]{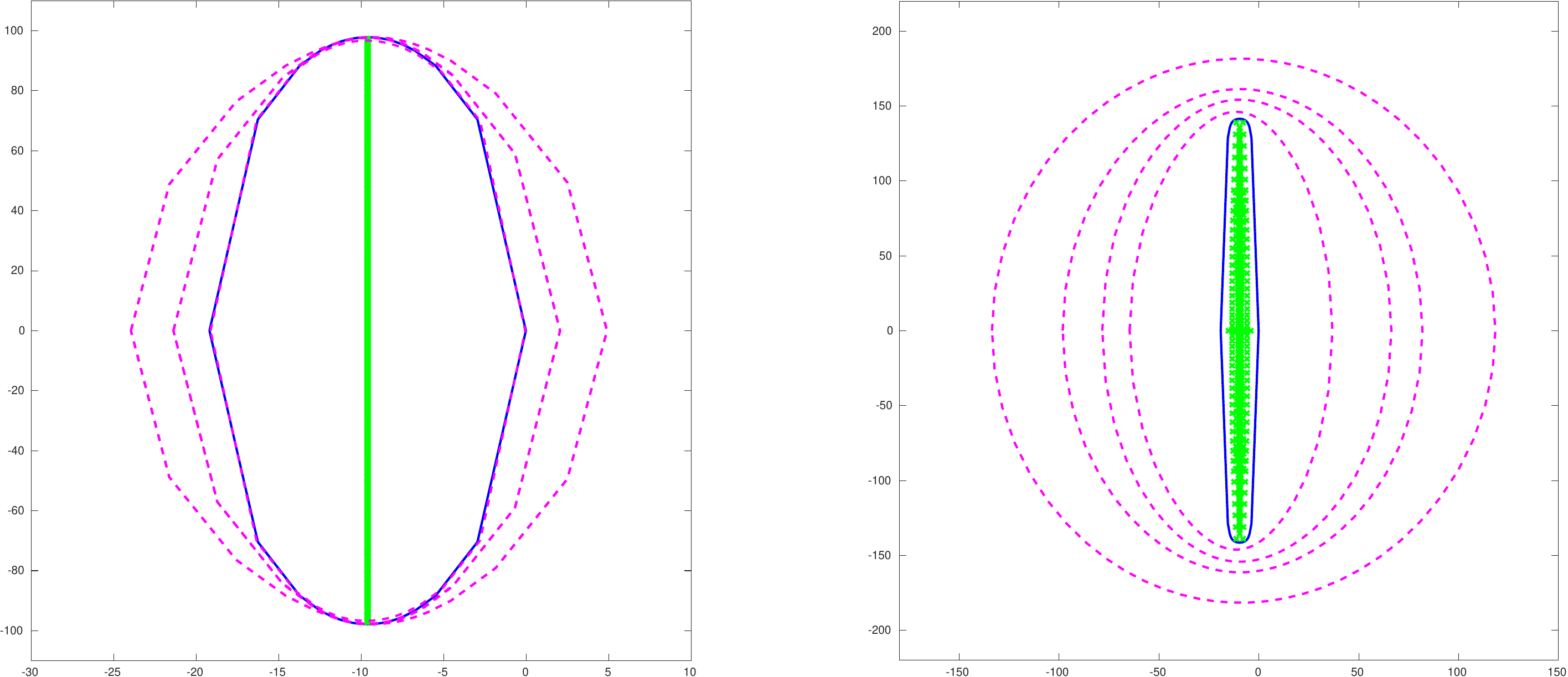}
\caption{\Cref{Example 2}. Field of values of the coefficient matrix (solid line), corresponding eigenvalues (approximately vertically aligned crosses) and field of values of the reduced matrices, for $d=10, 20, 30, 40$ and $r=1$ (expanding dashed lines as $d$ grows). Left: $\mat{A}$ and  $\widehat{\trunchessenberg{H}}_d+\widehat{\blockvec{H}}\blockvec{E}_d^{\top}$. Right: $\mat{B}$ and $\widehat{\trunchessenberg{G}}_d+\widehat{\blockvec{G}} \blockvec{E}_d^{\top}$. }\label{fig:fov_ex1}
\end{figure}

{\color{black} In \Cref{fig:fov_ex1} we report the boundary of the field of values of the two coefficient matrices (solid line, left plot for $\mat{A}$ and right plot for $\mat{B}$), together with those of the corresponding matrices $\widehat{\trunchessenberg{H}}_d+\widehat{\blockvec{H}}\blockvec{E}_d^{\top}$
and $\widehat{\trunchessenberg{G}}_d+\widehat{\blockvec{G}} \blockvec{E}_d^{\top}$, for $d=10, 20, 30, 40$ {\color{black}and $r=1$} (dashed lines). To be able to approximate the field of values of  $\mat{A}$ and $\mat{B}$ a coarse discretization of the two operators was considered, leading to $n=2500$. In light of this, we also set maxit$=100$ and $s=200$. We can appreciate that in both cases, and in particular for the matrix $\mat{B}$, the field of values of the reduced matrix is significantly larger than that of the corresponding coefficient matrix, and it is expanding with growing $d$. {\color{black} In principle, this behavior can potentially jeopardize the performance of \sPKSM. However, the following numerical results will show that \sPKSM~favorably compares to full Arnoldi. This supports our theoretical findings which indeed suggest that the field of values of the projected matrices is not the right tool to predict the convergence of the solver in the sketched Krylov framework.}

To analyze the computational performance of the proposed method, we} {\color{black}consider $n=90\,000$ and start by comparing the full Arnoldi method and \sPKSM. The threshold for the stopping criterion in full Arnoldi and \sPKSM~is $\tol=10^{-6}$.
The truncation parameter in \sPKSM~is set to $k=10$, the sketching dimension is $s=1\,600$ and maxit$=800$.}

{\color{black}In both procedures, the cost of solving the projected equations becomes more and more dominant as the dimensions of the reduced matrices increase, being of the order of {\color{black}$(rd)^3$}. To mitigate this shortcoming, it is common practice to solve the projected problem and check the stopping criterion every $p\geq 1$ iterations. We consider $p \in \{1,20\}$.}

{\color{black}
\begin{table}[t]
{\color{black}
\setlength{\tabcolsep}{2.5pt}
    \centering
    \begin{tabular}{r|rrr|rrr|rrr|rrr}
    & \multicolumn{6}{c}{Full Arnoldi} & \multicolumn{6}{|c}{\sPKSM}\\
    & \multicolumn{3}{c}{$p=1$} & \multicolumn{3}{c|}{$p=20$}    & \multicolumn{3}{c}{$p=1$} & \multicolumn{3}{c}{$p=20$}\\
    
         $\nu$  & $d$ & Time & Mem & $d$ & Time & Mem & $d$ & Time & Mem & $d$ & Time & Mem  \\
         \hline
         \hline
         \multicolumn{13}{c}{$r=1$}\\
         \hline
         \hline
         $0.1$ &   460 & 117.69 & {\it 922} & 460 & 90.49 & {\it 922} & 458 & 44.10 & {\it 56}  & 480 & 15.96  & {\it 54} \\
         $0.01$ &   553 & 185.65 & {\it 1\,108} & 560 & 133.50 & {\it 1\,122} & 550 & 77.80 & {\it 84} &  560 & 20.56 & {\it 84}\\
         $0.001$ & 745 & 409.26 & {\it 1\,492} & 760 & 255.58 & {\it 1\,522} & 770 & 228.39 & {\it 174} & 780 & 38.38 & {\it 174} \\
\hline
         \hline
         \multicolumn{13}{c}{$r=3$}\\
         \hline
         \hline
                 $0.1$ &   391 & 455.61 & {\it 2\,352} & 400 & 304.81 & {\it 2\,406} & 386 & 242.19 & {\it 156}  & 400 & 72.25  & {\it 156} \\
         
         $0.01$ &  501 & 856.75 & {\it 3\,012} & 520 & 534.57 & {\it 3\,126} & 496 & 504.37 & {\it 238} &  500 & 112.88 & {\it 238}\\

         $0.001$ & 694    & 2474.43 &  {\it 4\,170}  & 700 &  1052.56& {\it 4\,206}& 741 & 2006.41  & {\it 516} & 760 & 276.30  &  {\it 518} \\
         
    \end{tabular}
    \caption{\Cref{Example 2}. Performance of the full Arnoldi method and {\color{black}\sPKSM} {\color{black}as the viscosity $\nu$ varies}. The reported CPU timings are in seconds {\color{black} and $d$ is the number of iterations at termination}. ``Mem'' indicates the number of vectors of length $n$ used. {\color{black} The stopping criterion is checked every $p$ iterations. $r$ is the right-hand side rank.}}
    \label{tab1:ex2}
    }
\end{table}
}

    

In \Cref{tab1:ex2} we report the number of iterations and CPU times of the full Arnoldi method and the sketched-and-truncated scheme, as $\nu$ varies, {\color{black} and $r=1,3$}.
The two schemes perform a very similar number of iterations for all tested values of $\nu$. This shows that on this example the basis truncation does not significantly affect (delay) convergence and the implicit use of the $\mat{S}^{\top}\!\mat{S}$-orthogonal bases $\mat{\widehat U}_d$ and $\mat{\widehat V}_d$ pays off. We recall that in {\color{black}\sPKSM} we check the $\mat{S}^{\top}\!\mat{S}$-norm of the residual matrix whereas full Arnoldi computes its Frobenius norm. Although rather close thanks to the $\varepsilon$-embedding property of the sketching, these two norms may slightly differ (cf.~\eqref{eq:sketched_res_norm_bound}), leading to a small gap in the number of iterations needed to achieve a prescribed level of accuracy in the two algorithms.

Thanks to an effective truncation strategy with $k=10$,
 {\color{black}\sPKSM} always requires less CPU time than full Arnoldi, for a similar number of iterations.
Moreover, the rather small dimension of the sketching ($s=1\,600$), limits the impact of the full orthogonalization of the sketched bases on the overall solution process. {\color{black}We can also observe that using a larger value of $p$ is much more convenient {\color{black}(up to 85\% savings)} for {\color{black}\sPKSM} than for full Arnoldi (see \Cref{tab1:ex2} for $p=20$). This is because in the latter, the cost of constructing a fully orthonormal basis still remains largely dominant.}

{\color{black} Our primary concern is memory allocations. We} stress that only about $k=10$ long vector allocations are required by the truncated approach for each space, as opposed to the full basis. For more on-point insight on the memory requirements,  \Cref{tab1:ex2} (column labeled ``Mem'') includes the number of vectors of length $n$ stored by the two algorithms.
With an eye to \Cref{tab_memory}, for full Arnoldi  memory allocations are mainly due to the bases, whereas for {\color{black}\sPKSM} they are mainly due to storing the low-rank factors of the computed solution.
The memory requirements of {\color{black}\sPKSM} are always one order of magnitude smaller than what is needed by full Arnoldi, making the sketched-and-truncated Krylov subspace method very competitive.

{\color{black}
We also point out that the likely ill-conditioning of the non-orthogonal bases $\mathbf{U}_d$ and $\mathbf{V}_d$ does not have a strong impact on the convergence of \sPKSM. In \Cref{fig:basiscond} we fix $\nu=0.001$, $k=10$, $r=1$, and plot $\mathtt u \cdot \kappa_2(\mathbf{U}_d)$ and $\mathtt u \cdot \kappa_2(\mathbf{V}_d)$, where $\mathtt u$ denotes the unit roundoff, along with $\|\mathbf{R}_d\|_F$. We can observe that both $\mathbf{U}_d$ and $\mathbf{V}_d$ become numerically singular way before convergence is reached. On the other hand, this potential issue does not significantly affect the trend of $\|\mathbf{R}_d\|_F$. This was also noticed in~\cite{GuettelSimunec2023}.} 

\begin{figure}
\centering
\includegraphics[width=3.80in,height=2.3in]{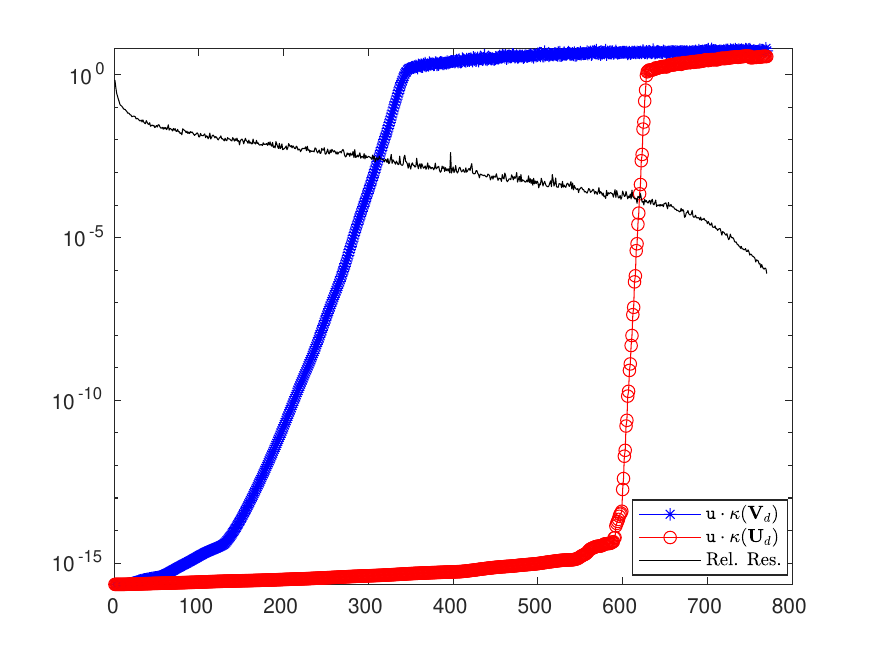}
\caption{\Cref{Example 2}. Scaled condition number of the bases $\mathbf{U}_d$ (red line with circles) and $\mathbf{V}_d$ (blue line with stars), along with the relative Frobenius norm of the residual matrix $\mathbf{R}_d$ (black solid line).}\label{fig:basiscond}
\end{figure}

We next turn our attention to the plain truncated scheme to analyze how different selections of the truncation parameters for the two spaces, namely $k_{\mat{A}}$ and $k_{\mat{B}}$, may affect performance.
In contrast to what happens for {\color{black}\sPKSM} (cf. \Cref{tab1:ex2} for $k_{\mat{B}}=10$ {\color{black}and $r=1$}), {\color{black}\tPKSM} was unable to achieve the prescribed level of accuracy for $k_{\mat{A}}=k_{\mat{B}}=10$ within maxit=1000 iterations for any considered value of $\nu$.
We thus fix $k_{\mat{A}}=40$ and in \Cref{tab2:ex2} we report the number of iterations and computational timings of {\color{black}\tPKSM} when varying $\nu$ and $k_{\mat{B}}$.
\begin{table}[t]
{\color{black}
\setlength{\tabcolsep}{1.7pt}
    \centering
    \begin{tabular}{r|rrr|rrr|r|rrr|rrr}
    & \multicolumn{6}{c}{Full Arnoldi} & \multicolumn{7}{|c}{\tPKSM}\\

    & \multicolumn{3}{c}{$p=1$} & \multicolumn{3}{c|}{$p=20$}
      &\multicolumn{1}{c}{ } & \multicolumn{3}{c}{$p=1$} & \multicolumn{3}{c}{$p=20$}
    \\
    
       $\nu$   & $d$ & Time& Mem & $d$ & Time & Mem & $k_{\mat{B}}$  & $d$ & Time & Mem & $d$ & Time & Mem \\
       \hline 
       \hline 
       \multicolumn{13}{c}{$r=1$, $k_\mathbf{A}=40$}\\
   \hline\hline
        \multirow{2}{*}{0.1} & \multirow{2}{*}{460}& \multirow{2}{*}{117.69}  & \multirow{2}{*}{{\it 922}} &
        \multirow{2}{*}{460}& \multirow{2}{*}{90.49} & \multirow{2}{*}{\it 922} & 40 &   565 & 107.11 & {\it 80} & 580 & 48.46& {\it 80} \\
        &&&&&  &&  60 &  601 & 146.25 &  {\it 100} & 620 & 75.77&  {\it 100}  \\
        \hline 
        \multirow{2}{*}{0.01}  &\multirow{2}{*}{553} & \multirow{2}{*}{185.65}& \multirow{2}{*}{\it 1\,108} & \multirow{2}{*}{560} & \multirow{2}{*}{133.50}&
        \multirow{2}{*}{\it 1\,122} & 40 &   671 & 168.58 &  {\it 98} & 680 & 63.01& {\it 98}  \\
        & & &&&& & 60 &  742 & 237.24 &  {\it 100} & 760 & 99.40 & {\it 100}  \\
       \hline 
       \hline 
       \multicolumn{13}{c}{$r=3$, $k_\mathbf{A}=100$}\\
   \hline\hline

   \multirow{2}{*}{0.1} & \multirow{2}{*}{391}& \multirow{2}{*}{455.61}  & \multirow{2}{*}{{\it 2\,346}} &
        \multirow{2}{*}{400}& \multirow{2}{*}{304.81} & \multirow{2}{*}{\it 2\,400} & 100 &   1196 & 11013.21 & {\it 600} & 1200 & 1534.59 & {\it 600} \\
        &&&&&  &&  150 & 839 & 3033.78 &  {\it 750} & 840 & 930.22&  {\it 750}  \\
        \hline 

        \multirow{2}{*}{0.01}  &\multirow{2}{*}{501} & \multirow{2}{*}{856.75}& \multirow{2}{*}{\it 3\,006} & \multirow{2}{*}{520} & \multirow{2}{*}{534.57}&
        \multirow{2}{*}{\it 3\,120} & 100 & 1069  & 6841.95  &  {\it 600} &1080 & 1203.43 &  {\it 600}   \\
        
        & & &&&& & 150 & 874  & 3529.69 &  {\it 750} & 880 & 1033.76 & {\it 750}  
            \end{tabular}
            
    \caption{\Cref{Example 2}. Performance of the full Arnoldi method and {\color{black}\tPKSM}. {\color{black} $k_{\mat{A}}$ and $k_{\mat{B}}$ denote the truncation parameters employed in \tPKSM~for the construction of the left and right space, respectively}. The reported CPU timings are in seconds. ``Mem'' indicates the number of vectors of length $n$ used.  }\label{tab2:ex2}
    }
    \end{table}
With these values of $k_{\mat{A}}$ and $k_{\mat{B}}$, {\color{black}\tPKSM} converges and
considerations about the use of $p\geq 1$ similar to the ones discussed for the sketched-and-truncated procedure still hold. Similarly for the memory requirements.

{\color{black} In \Cref{tab2:ex2} we also report the results achieved by setting $r=3$. We notice that in this case larger values of $k_{\mat{A}}$ and $k_{\mat{B}}$ are needed to be able to meet the required level of accuracy in 2\,000 iterations. In particular, we used $k_{\mat{A}}=100$ and $k_{\mat{B}}=100,150$. In spite of these sizable truncation parameters, \tPKSM~still shows a significant delay in the number of iterations when compared to full Arnoldi. This drawback makes \tPKSM~less appealing in terms of computational time.

Due to the large values of $k_{\mat{A}}$ and $k_{\mat{B}}$, also the memory requirements of \tPKSM~increase. Indeed, in general, these no longer amount to the allocation of the low-rank factors of the approximate solution as in \sPKSM, but rather to the storage of the truncated bases, namely $(k_{\mat{A}}+k_{\mat{B}})r$ vectors of length $n$. 
Nevertheless, the storage demand of \tPKSM~is still remarkably smaller than that of full Arnoldi for both $r=1,3$.}

To conclude this example, we would like to highlight a somehow unexpected phenomenon: 
{\color{black}\tPKSM} performs more iterations for $k_{\mat{B}}=60$ than for $k_{\mat{B}}=40$, in spite of the former having more orthogonal basis vectors.
This peculiar aspect shows that selecting sensible values of $k$ for {\color{black}\tPKSM} may be a rather difficult task. Further analysis is needed to fully understand this behavior.
\end{example}

 \begin{example}\label{Example_heat}
 For this example we take inspiration from~\cite[Example 5.1]{Henningetal2021}. We consider the heat equation $u_t + \Delta u = f$ in the cube $(-1,1)^3$ and time interval $[0,1]$, with homogeneous Dirichlet boundary conditions and zero initial value. 
 The source term is $f (t, x, y, z) :=
10 \, t\sin(t) \cos( 2\pi x) \cos( 2\pi y) \cos( 2\pi z)$ so that its space-time discretized form has rank equal to one.
The discretization  of this differential problem leads to the following Sylvester equation
\begin{equation}\label{eq:ex_heat}
    \mat{K} \mat{X}\mat{B}_1 + \mat{M}\mat{X}\mat{B}_2 = c_1c_2^\top,
\end{equation}    
where $\mat{B}_1$, $\mat{B}_2\in\mathbb{R}^{\ell \times \ell}$, with $\ell$ the number of time steps, are such that 
$$\mat{B}_1=\frac{1}{2(\ell-1)}\begin{bmatrix}
1 & & & \\
1 & 1 & & \\
& \ddots & \ddots & \\
& & 1 & 1\\
\end{bmatrix},\quad \mat{B}_2=\begin{bmatrix}
1 & & & \\
-1 & 1 & & \\
& \ddots & \ddots & \\
& & -1 & 1\\
\end{bmatrix},
$$
and correspond to the discrete Crank-Nicolson time integrator,
whereas $\mat{K}$, $\mat{M}\in\mathbb{R}^{n\times n}$ are, respectively, the stiffness and mass matrix coming from the discretization of the space component of the differential operator by linear finite elements; see~\cite{Henningetal2021} for more details.

We recast equation~\eqref{eq:ex_heat} in terms of a standard Sylvester equation by inverting $\mat{M}$ and $\mat{B}_1$. In particular, we consider the Sylvester equation
\begin{equation}\label{eq:ex_heat2}
    \mat{M}^{-1}\mat{K} \mat{X} + \mat{X}\mat{B}_2\mat{B}_1^{-1} = \mat{M}^{-1}c_1c_2^\top\mat{B}_1^{-1}.
\end{equation}    
Thanks to the moderate number of time steps $\ell$ we consider ($\ell=800$) and the sparsity pattern of $\mat{B}_1$, we can explicitly construct\footnote{{\color{black}This operation is performed by means of the Matlab routine {\tt mrdivide}.} On the other hand, by using a specifically adapted procedure, the solution of the reduced Sylvester equation could be performed without explicitly forming $\mat{B}_2\mat{B}_1^{-1}$; this fact should be considered for finer time discretizations.} the matrix $\mat{B}_2\mat{B}_1^{-1}$.  On the other hand, the mass matrix $\mat{M}$ inherits the possibly involved sparsity pattern of $\mat{K}$. Moreover, we employ large values of degrees of freedom $n$ in space so that $\mat{M}^{-1}\mat{K}$ cannot be explicitly computed. The action of this matrix is computed on the fly during the construction of the Krylov subspace $\mathcal{K}_d(\mat{M}^{-1}\mat{K}, \mat{M}^{-1}c_1)$. In particular, quantities of the form $\mat{M}^{-1}v$ are approximated by applying the preconditioned conjugate gradient method (PCG)~\cite{Hestenes1952} with diagonal (Jacobi) preconditioning. PCG is stopped as soon as the relative residual norm is smaller than $10^{-8}$.

Following~\cite{Henningetal2021}, equation~\eqref{eq:ex_heat2} can be solved by performing only a left projection, that is, only the left space $\mathcal{K}_d(\mat{M}^{-1}\mat{K}, \mat{M}^{-1}c_1)$ is constructed. The projected problem has the form
$$ \mat{\widetilde H}_d\mat{Y}+\mat{Y}\mat{B}_2\mat{B}_1^{-1}=e_1\mat{\beta}_1c_2^\top\mat{B}_1^{-1},$$
where $\mat{\widetilde H}_d\in\mathbb{R}^{d\times d}$ stems from the adopted projection technique, namely full Arnoldi ($\mat{\widetilde H}_d=\fullhessenberg{H}_d$), {\color{black}\tPKSM} ($\mat{\widetilde H}_d=\mat{H}_d$), and {\color{black}\sPKSM} ($\mat{\widetilde H}_d=\mat{\widehat H}_d+\widehat He_d^\top$).
\begin{table}[t]
    \centering
        {
    \begin{tabular}{r|rrr|rrr|rrr}
    \hline
    & \multicolumn{3}{c}{Full Arnoldi} & \multicolumn{3}{|c}{{\color{black}\tPKSM}}& \multicolumn{3}{|c}{{\color{black}\sPKSM}} 
    \\
    
         $n$  & $d$ & Time & Mem & $d$ & Time & Mem & $d$ & Time & Mem 
         \\
         \hline
         $393968$ &   260 & 100.93 &  {\it 261} & 500 & 172.96 &  {\it 20} & 260 & 84.57 & {\it 16} 
         \\
         $482768$ &    { 280}& 145.43 & {\it 281}  & 780 & 317.47 & {\it 22}  & 280 & 131.25 & {\it 16} 
         \\
         $685214$ &    { 300}& 224.10 &  {\it 301} & 720 & 352.41 & {\it 22}  & 280 & 196.65 & {\it 16} 
         \\
         $880370$ &    {\ 320}& 327.02 &  {\it 321} & 860 & 531.06 &  {\it 22} & 340 & 258.65 & {\it 16}
         \\        
         \hline
           \end{tabular}
    \caption{\Cref{Example_heat}. Performance of all considered methods for varying $n$ ($\ell=800$, truncation $k=3$). The reported CPU timings are in seconds. ``Mem'' indicates the number of vectors of length $n$ used. }
    \label{tab1:ex4}
    }
\end{table}

We compare the three schemes as $n$ varies, in terms of number of iterations and computational time, by checking the stopping criterion every $p=20$ iterations, and {$\mathtt{tol}=10^{-6}$}. 
For the sketched-and-truncated scheme we adopt $s=1200$; the truncation parameter takes the value $k=3$ in both {\color{black}\sPKSM} and {\color{black}\tPKSM}.

The results displayed in~\Cref{tab1:ex4} show that full Arnoldi and {\color{black}\sPKSM} perform a rather similar number of iterations for all tested values of $n$
with {\color{black}\sPKSM} being always faster than full Arnoldi.
 On the contrary, 
{\color{black}\tPKSM} needs a much larger number of iterations to converge, at least twice as many as with {\color{black}\sPKSM}, leading to correspondingly large CPU time.

For this problem, the real potential of {\color{black}\sPKSM} and {\color{black}\tPKSM} can be appreciated by looking at their memory consumption. Indeed, both 
schemes require storing one order of magnitude fewer vectors of length $n$ than what is done in full Arnoldi, making the memory requirements of the former solvers extremely affordable.
\end{example}

\begin{example}\label{Example 3}
In this example we compare the sketched-and-truncated strategy with the rational Krylov subspace method (RKSM); see, e.g.~\cite{DRUSKIN2011546}.
We consider the matrices $\mat{A},\mat{B}\in\mathbb{R}^{n\times n}$ in~\eqref{eq:sylvester_equation} coming from the centered finite differences discretization of three-dimensional convection-diffusion operators of the form~\eqref{eq:condiff_op} on $[0,1]^3$ with $\nu=0.005$ and convection vectors $w_{\mat{A}}=(x\sin(x),y\cos(y),e^{z^2-1})$ and $w_{\mat{B}}=((1-x^2)yz,1,e^z)$. The vectors $c_1,c_2\in\mathbb{R}^n$ have normally distributed random entries, and they are such that $c_1c_2^{\top}$ has unit norm.

The shifts employed in RKSM are computed by following the on-line procedure presented in~\cite{DRUSKIN2011546}. Moreover, the large shifted linear systems occurring in the basis construction are solved by means of preconditioned BICGstab($\ell$)~\cite{VanDerVorst1992}
with no fill-in ILU preconditioning\footnote{We employed the Matlab function {\tt bicgstabl}, and for the preconditioner, {\tt setup=struct('type','nofill'); [LA,UA]=ilu(A,setup).}}, and similarly for $\mat{B}$. BICGstab($\ell$) is stopped as soon as the relative residual norm reached $10^{-8}$.

The stopping threshold in {\color{black}\sPKSM} and RKSM is $\tol=10^{-6}$ while maxit$=250$.
Finally, in {\color{black}\sPKSM} the truncation parameter was $k=3$ while the sketching dimension was $s=500$.

\begin{table}[t]
    \centering
    \begin{tabular}{r|rrrr|rrrr}
    \hline
    &  \multicolumn{4}{c}{{\color{black}\sPKSM}} & \multicolumn{4}{|c}{RKSM} \\
         $n$  & $d$ & Time & rank  & Mem & $d$ ($d_{\textnormal{inn}}^{\mat{A}}$, $d_{\textnormal{inn}}^{\mat{B}}$) & Time  & rank & Mem  \\
         \hline
         $125\,000$ &   140 & 2.69 & 29 & {\it 58} & 28 (7.4, 6.9) & 9.86 & 27 & {\it 58+36}  \\
         $216\,000$ &  160 & 5.36 & 30 & {\it 60} & 28 (7.1, 6.9) & 17.57 & 27 & {\it 58+36}  \\
         $343\,000$ &   180 & 11.29 & 31 & {\it 62} & 29 (7.8, 7.3) & 30.83 & 27 & {\it 60+36}\\
         $512\,000$ &  180 & 21.42 & 32 &  {\it 64} & 29 (8.2, 7.7) & 48.92 & 27 & {\it 60+36}\\
         $729\,000$ & 200 & 35.42 & 32& {\it 64} & 29 (8.8, 8.1) & 76.60 & 27 &  {\it 60+36}\\
         $1\,000\,000$ & 220 & 52.56 & 33 & {\it 66} & 29 (9.6, 8.7) & 114.91 & 29 & {\it 60+36}\\
         \hline
    \end{tabular}
    \caption{\Cref{Example 3}. Performance of {\color{black}\sPKSM} and RKSM. The reported CPU timings are in seconds. The column ``rank'' indicates the rank of the solution matrix $\mat{X}_d$ and ``Mem'' indicates the number of vectors of length $n$ used. For RKSM, the memory requirements are reported as the sum of the memory required for basis allocation and for the preconditioner.}
    \label{tab1:ex3}
\end{table}

In \Cref{tab1:ex3} we report the number of iterations, the computational timings and the rank of the solution computed by {\color{black}\sPKSM} and RKSM by varying the problem dimension $n$.
\Cref{tab1:ex3} also reports the average number of (inner) BICGStab($\ell$) iterations, $d_{\text{inn}}^{\mat{A}}$ and $d_{\text{inn}}^{\mat{B}}$, to approximately solve the shifted linear systems with $\mat{A}$ and $\mat{B}^\top$, respectively, during the generation of the rational Krylov spaces.
Although RKSM requires a much smaller space to converge, {\color{black}\sPKSM} is always faster. Indeed, the solution of two shifted linear systems (one with $\mat{A}$ and one with $\mat{B}^\top$) at each RKSM iteration has a detrimental impact on the overall running time.
Even though these values remain rather moderate, the BICGStab($\ell$) iteration count increases with $n$, making the linear system solution step increasingly more expensive. Such shortcoming may be fixed by employing more robust preconditioning operators.
On the other hand, {\color{black}\sPKSM} requires only matrix-vector products and the cost of the orthogonalization step is moderate thanks to the truncated approach. Moreover, the possible delay coming from the use of locally orthogonal bases is mitigated by sketching.

{\color{black}\sPKSM} turns out to be surprisingly competitive when compared to RKSM also in terms of storage allocation. Indeed, as reported in the columns ``Mem'' of \Cref{tab1:ex3}, the memory requirements of {\color{black}\sPKSM} and RKSM are always rather similar. We mention that in case of RKSM we denote the storage demand as the sum of two terms to emphasize the different nature of the allocated components. While the first one is related to the basis storage, the second one comes from the allocation of the preconditioners.
 Indeed, for this example, storing the four incomplete LU factors is equivalent to storing about 36 vectors of length $n$.
\end{example}

\section{Conclusions}\label{sec:conclusions}
By fully exploiting the promising combination of sketching and Krylov methods, in this paper we developed a sketched-and-truncated Krylov subspace method for matrix equations. In addition to making polynomial Krylov methods feasible, the approach turns out to be competitive also with respect to state-of-the-art schemes based on rational Krylov subspaces. In particular, the proposed method reduces the storage demand of the overall solution process, and concurrently provides a 50\% CPU time speed-up over RKSM. Therefore, the sketched-and-truncated polynomial Krylov subspace method can be a valid candidate for solving large-scale matrix equations for which projection methods based on more sophisticated subspaces struggle due to, e.g., expensive linear systems solves.

The whitened-sketched recurrence produces matrices that can, in theory, endanger the applicability of the generated space as projection method; fortunately, the projected linear equations are usually solvable in practice, and very effective performance resulted in our numerical experiments. We have theoretically explained this perceivable contradiction. Indeed, we have proved that sketching takes place into what we have called the effective field of values of the involved matrices, for which the projected solution can be well-defined.

A convergence analysis for the sketched approximation has also been presented, which conveniently extends results obtained in the literature for the full Arnoldi method.

As a supplementary numerical contribution of this paper, we have derived and analyzed truncated Krylov methods, which appear to be new in the matrix equation context.

\section*{Acknowledgments}
The first and third authors are members of the INdAM Research
Group GNCS that partially supported this work through the funded project GNCS2023 ``Metodi avanzati per la risoluzione di PDEs su griglie strutturate, e non'' with reference number CUP\_E53C22001930001.

\appendix
\section{Algorithmic description of truncated Arnoldi}\label{appendix:algorithm}
\begin{algorithm}[t]
\caption{Truncated Arnoldi method for Sylvester equations {\color{black}(\tPKSM)}}\label{alg:standardinnerprod_krylov}
\begin{algorithmic}[1]
\setstretch{1.2}
\smallskip
\Statex \textbf{Input:} \ \ $\mat{A}\in\R^{n\times n}$, $\mat{B}\in \R^{n \times n}$,  $\blockvec{C}_1 \in \R^{n\times r}$, $\blockvec{C}_2 \in \R^{n\times r}$, integers $0<k\leq\text{maxit} \ll n$, $\tol>0$, $p\geq 1$ 
\Statex \textbf{Output:} $\mat{X}^{(1)},\mat{X}^{(2)}$ such that $\mat{X}^{(1)}(\mat{X}^{(2)})^{\top}=\mat{X}_d$ approximately solves $\mat{A} \mat{X}+\mat{X}\mat{B}=\blockvec{C}_1\blockvec{C}_2^{\top}$
\State Compute skinny QRs: $\blockvec{U}_1\blockscalar{\beta}_1=\blockvec{C}_1$, $\blockvec{V}_1\blockscalar{\beta}_2=\blockvec{C}_2$
\For{$d=1,\ldots,maxit$}
\State Compute $\,\,\widetilde{\blockvec{U}}=\mat{A} \blockvec{U}_d$, $\,\,\widetilde{\blockvec{V}}=\mat{B}^{\top}\blockvec{V}_d$
\For{$i=\max\{1,d-k+1\},\ldots,d$}
\State Set $\widetilde{\blockvec{U}}= \widetilde{\blockvec{U}}-\blockvec{U}_i\blockscalar{h}_{i,d}\,\,$
with $\,\,\blockscalar{h}_{i,d}=\blockvec{U}_i^{\top}\widetilde{\blockvec{U}}$
\State Set $\widetilde{\blockvec{V}}= \widetilde{\blockvec{V}}-\blockvec{V}_i\blockscalar{g}_{i,d}\,\,$
with $\,\,\blockscalar{g}_{i,d}=\blockvec{V}_i^{\top}\widetilde{\blockvec{V}}$
\EndFor
\State Compute skinny QRs:  $\blockvec{U}_{d+1}\blockscalar{h}_{d+1,d}=\widetilde{\blockvec{U}}$ and $\blockvec{V}_{d+1}\blockscalar{g}_{d+1,d}=\widetilde{\blockvec{V}}$
\If{$\mathtt{mod}(d,p)=0$}
\State Solve $\,\,\trunchessenberg{H}_d \mat{Y} + \mat{Y} \trunchessenberg{G}_d^{\top} = \blockvec{e}_1\blockscalar{\beta_1}\blockscalar{\beta_2}\blockvec{e}_1^{\top}$ for $\mat{Y}$
\State Compute $\rho=\sqrt{dr}(\|{\color{black}\mat{Y}}E_{d}\blockscalar{g}_{d+1,d}^\top\|_F+ \|\blockscalar{h}_{d+1,d}\blockvec{E}_d^{\top}{\color{black}\mat{Y}}\|_F)$~\label{alg:line_rho}
\If{$\rho<\tol$}
\State Go to line~\ref{line:endfor}
\EndIf
\EndIf
\EndFor
\State Compute (possibly low-rank) factors $\mat{Y}_1, \mat{Y}_2$ such that $\mat{Y}\approx\mat{Y}_1\mat{Y}_2^{\top}$\label{line:endfor}
\State Retrieve $\mat{X}^{(1)}= \truncbasis{U}_d\mat{Y}_1$, $\mat{X}^{(1)}=\truncbasis{V}_d\mat{Y}_2$ by the two-pass step
\end{algorithmic}
\end{algorithm}
In~\Cref{alg:standardinnerprod_krylov} we report the algorithm pseudocode for the truncated polynomial Krylov method for Sylvester equations. Notice that, in principle, one can use the sketched inner product also in the computation of the truncated bases $\truncbasis{U}_d$ and $\truncbasis{V}_d$, thus avoiding any inner product in $\mathbb{R}^n$. We refrain from reporting also this pseudocode here as it can be easily derived from \Cref{alg:standardinnerprod_krylov}.

\section{Tensorized subspace embeddings}\label{appendix:embedding}
In our context of solving matrix equations, sketchings of the form $\mat{S}_{\mat{U}}\mat{M}\mat{S}_{\mat{V}}^\top$ occur; see, e.g., formula~\eqref{eq:resnorm_sk} involving the norm of the sketched residual.

The embeddings $\mat{S}_{\mat{U}}$ and $\mat{S}_{\mat{V}}$ are constructed as usual subspace embeddings for the block Krylov spaces $\spK_{d+1}(\mat{A},\blockvec{C}_1), \spK_{d+1}(\mat{B}^\top,\blockvec{C}_2)$, so that~\eqref{eq:sketch} holds for $\mat{S}_{\mat{U}}$ and $v \in \spK_{d+1}(\mat{A},\blockvec{C}_1)$ as well as for $\mat{S}_{\mat{V}}$ and $v \in \spK_{d+1}(\mat{B}^\top,\blockvec{C}_2)$. We now investigate what this implies for 
$\|\mat{S}_{\mat{U}}\mat{M}\mat{S}_{\mat{V}}^\top\|_F$ in relation to $\|\mat{M}\|_F$, when $\mat{M}$ is of the form $\mat{M} = \fullbasis{U}_{d+1}\fullhessenberg{Z}\fullbasis{V}_{d+1}^\top$ (as it is the case for the Sylvester equation residual).

Our derivations closely follow the ideas in~\cite{chen2020structured}, generalizing their results from random Gaussian matrices to general $\varepsilon$-subspace embeddings. We start with the following simple observation.

\begin{proposition}\label{pro:SV_embedding}
{\it Let $\mat{S} \in \mathbb{R}^{s \times n}$ be an $\varepsilon$-subspace embedding of $\mathcal{V}$ with $\dim(\mathcal V) = d+1$ and let the columns of  $\fullbasis{V}$ form an orthonormal basis of $\mathcal{V}$. Then $\mat{S}\fullbasis{V}$ is an $\varepsilon$-subspace embedding for $\mathbb{R}^{d+1}$.
}
\end{proposition}
\begin{proof}
Let $z \in \mathbb{R}^{d+1}$. Then $\fullbasis{V}z \in \mathcal{V}$, so that by~\eqref{eq:sketch}, we have
\[
(1-\varepsilon) \| \fullbasis{V}z \|^2 \leq \| \mat{S}\fullbasis{V}z\|^2 \leq (1+\varepsilon) \|\fullbasis{V}z\|^2,
\]
The result follows by noting that $\|\fullbasis{V}z\| = \|z\|$ due to the orthonormality of the columns of $\fullbasis{V}$.
\end{proof}

We further need the following auxiliary result.
\begin{proposition}[Lemma~1 and Corollary~3 in~\cite{chen2020structured}]\label{pro:kron_embedding}
{\it
Let $\mat{S}_1, \mat{S}_2 \in \mathbb{R}^{s \times (d+1)}$ be $\varepsilon$-subspace embeddings of $\mathbb{R}^{d+1}$. Then $\mat{S}_i \otimes \mat{I}_n$ and $\mat{I}_n \otimes \mat{S}_i, i = 1,2$ are $\varepsilon$-subspace embeddings of $\mathbb{R}^{n(d+1)}$ and $\mat{S}_1 \otimes \mat{S}_2$ is an $\varepsilon(2+\varepsilon)$-subspace embedding of $\mathbb{R}^{(d+1)^2}$.
}
\end{proposition}

With these results, we can now state a tensorized embedding property.
\begin{theorem}\label{thm:kron_embedding}
Let $\mat{S}_{\mat{U}} \in \mathbb{R}^{s \times n_1}$ be an $\varepsilon$-subspace embedding for the subspace $\mathcal{U} \subseteq \mathbb{R}^{n_1}$ and $\mat{S}_{\mat{V}}  \in \mathbb{R}^{s \times n_2}$ be an $\varepsilon$-subspace embedding for the subspace $\mathcal{V} \subseteq \mathbb{R}^{n_2}$, with $\dim(\mathcal{U})=\dim(\mathcal{V}) = d+1$. Further, let the columns of $\fullbasis{U}, \fullbasis{V}$ contain orthonormal bases of $\mathcal{U}, \mathcal{V}$, respectively.

Then, for any $\mat{M} = \fullbasis{U}\fullhessenberg{Z}\fullbasis{V}^\top$, we have
\begin{equation}\label{eq:sketchM}
    (1-\widetilde{\varepsilon}) \|\mat{M}\|_F \leq \|\mat{S}_{\mat{U}}\mat{M}\mat{S}_{\mat{V}}^\top\|_F \leq (1+\widetilde{\varepsilon}) \|\mat{M}\|_F
\end{equation}
with $\widetilde{\varepsilon} = \varepsilon(2+\varepsilon)$.
\end{theorem}
\begin{proof}
Using the vectorized form, for any $\mat{M}$, we have $\|\mat{M}\|_F = \|\vecop(\mat{M})\|$. By a direct calculation, we obtain 
\begin{align*}
\|\mat{S}_{\mat{U}}\mat{M}\mat{S}_{\mat{V}}^\top\|_F  =  \|(\mat{S}_{\mat{V}}\fullbasis{V} \otimes \mat{S}_{\mat{U}}\fullbasis{U}) \vecop(\fullhessenberg{Z})\|.
\end{align*}
By \Cref{pro:SV_embedding}, $\mat{S}_{\mat{U}}\fullbasis{U}$ and $\mat{S}_{\mat{V}}\fullbasis{V}$ are $\varepsilon$-subspace embeddings of $\mathbb{R}^{d+1}$, so that by \Cref{pro:kron_embedding}, $\mat{S}_{\mat{V}}\fullbasis{V} \otimes \mat{S}_{\mat{U}}\fullbasis{U}$, is an $\widetilde{\varepsilon}$-subspace embedding of $\mathbb{R}^{(d+1)^2}$. Thus,
\begin{equation}\label{eq:sketched_fro}
(1-\widetilde{\varepsilon})\|\vecop(\fullhessenberg{Z})\| \leq \|(\mat{S}_{\mat{V}}\fullbasis{V} \otimes \mat{S}_{\mat{U}}\fullbasis{U}) \vecop(\fullhessenberg{Z})\| \leq (1+\widetilde{\varepsilon})\|\vecop(\fullhessenberg{Z})\|.
\end{equation}
Furthermore, $\|\vecop(\fullhessenberg{Z})\| = \|\vecop(\mat{M})\|$, because $(\fullbasis{V} \otimes \fullbasis{U})\vecop(\fullhessenberg{Z}) = \vecop(\mat{M})$ and $\fullbasis{V} \otimes \fullbasis{U}$ has orthonormal columns. Inserting this into~\eqref{eq:sketched_fro} completes the proof.
\end{proof}

\begin{remark}
The lower bound in~\eqref{eq:sketchM} is only informative if $1-\widetilde{\varepsilon} > 0$, or equivalently, $\varepsilon < \sqrt{2} - 1 \approx 0.4142$. This means that in order to obtain theoretical guarantees, one needs to impose a stricter subspace embedding condition than when solving linear systems or approximating the action of matrix functions via sketched Krylov methods (where any $\varepsilon < 1$ works in principle, and $\varepsilon = 1/\sqrt{2}$ guarantees that, e.g., the sketched residual norm differs from the true residual norm by a factor of at most $6$). 

However, our numerical experiments indicate that in practice, a cruder tolerance typically also suffices for obtaining a method that works well.
\end{remark}

\bibliographystyle{amsplain}
\bibliography{refs}

\end{document}